\documentclass[11pt, letterpaper,final]{amsart}
\usepackage{amsfonts,amsmath, amsthm, amssymb, latexsym, epsfig}
\usepackage[english]{babel}
\usepackage[utf8]{inputenc}
\usepackage[all]{xy}
\usepackage{xspace}
\usepackage{amscd}
\usepackage{comment}
\usepackage{setspace}
\usepackage{enumerate}
\usepackage{stmaryrd}
\usepackage{xcolor}
\usepackage{tikz-cd}
\usepackage{hyperref}
\usepackage[mathscr]{eucal}
\usepackage[notcite,notref]{showkeys}

\usepackage[hmargin=3cm,vmargin=3cm]{geometry}

\newcommand{\A}{{\mathcal{A}}}
\newcommand{\B}{{\mathcal{B}}}
\newcommand{\C}{{\mathcal{C}}}
\newcommand{\D}{{\mathcal{D}}}

\renewcommand{\H}{{\mathcal{H}}}

\newcommand{\K}{{\mathcal{K}}}
\renewcommand{\L}{{\mathcal{L}}}
\newcommand{\M}{{\mathcal{M}}}

\renewcommand{\O}{{\mathcal{O}}}

\newcommand{\T}{{\mathcal{T}}}
\newcommand{\U}{{\mathcal{U}}}

\newcommand{\bT}{\mathbb{T}}

\newcommand{\bbC}{{\mathbb{C}}}

\newcommand{\bbN}{{\mathbb{N}}}

\newcommand{\bbT}{{\mathbb{T}}}
\newcommand{\bbZ}{{\mathbb{Z}}}

\newcommand{\id}{\operatorname{id}}
\newcommand{\lag}{\operatorname{lag}}

\newcommand{\ev}{\operatorname{ev}}
\newcommand{\Aut}{\operatorname{Aut}}

\newtheorem{thm}{Theorem}[section]
	\newtheorem{cor}[thm]{Corollary}
	\newtheorem{lem}[thm]{Lemma}
	\newtheorem{prp}[thm]{Proposition}
	\newtheorem{conj}[thm]{Conjecture}

	\newtheorem{theoremintro}{Theorem}

\theoremstyle{definition}
	\newtheorem{dfn}[thm]{Definition}
	
	\newtheorem{asmp}[thm]{Assumption}
	\newtheorem{rmk}[thm]{Remark}

\begin{document}
	\title[Equivariant homotopy classification of graph C*-algebras]{Equivariant homotopy classification \\ of graph C*-algebras}
	
	        \author{Boris Bilich}
        \address{Department of Mathematics, University of Haifa, Haifa, Israel, and Department of Mathematics, University of G\"ottingen, G\"ottingen, Germany.}
        \email{bilichboris1999@gmail.com}
        
        \author{Adam Dor-On}
        \address{Department of Mathematics, University of Haifa, Mount Carmel, Haifa 3103301, Israel.}
        \email{adoron.math@gmail.com}
				
				\author{Efren Ruiz}
        \address{Department of Mathematics\\University of Hawaii,
Hilo\\200 W. Kawili St.\\
Hilo, Hawaii\\
96720-4091 USA}
        \email{ruize@hawaii.edu}
        \date{\today}
	\subjclass[2020]{}
	

	
	\thanks{B. Bilich was partially supported by a Bloom PhD scholarship at Haifa University. A. Dor-On was partially supported by an NSF-BSF grant no. 2530543 / 2023695 (respectively), a Horizon Marie-Curie SE project no. 101086394 and a DFG Middle-Eastern collaboration project no. 529300231. E. Ruiz was partially supported by the Simons Foundation and by the Mathematisches Forschungsinstitut Oberwolfach.}
        
\begin{abstract}
We show that shift equivalence of essential adjacency matrices coincides with gauge-equivariant homotopy equivalence of their stabilized graph C*-algebras. This provides the first equivalent formulation of shift equivalence of essential matrices in terms of gauge actions on graph C*-algebras. Our proof uses bicategory theory for C*-bimodules developed by Meyer and Sehnem, allowing us to avoid the use of K-theory classification of C*-algebras. 
\end{abstract}
        
\maketitle

\section{Introduction}

Motivated by the decidability problem for conjugacy of subshifts of finite type (SFTs) \cite{KRdeci, KRdecii}, as well as Williams conjugacy problem \cite{Wil73, KRred, KR99}, steady progress has been made in the past four decades in encapsulating various notions of equivalence between SFTs in terms of Cuntz-Krieger graph C*-algebras \cite{CK80, Kri80, Cu81, Mat10, Mat17, ERRS21}. In his seminal work, Williams recaptured conjugacy and eventual conjugacy of SFTs in terms of shift equivalence relations between associated adjacency matrices of directed graphs. These are known as strong shift equivalence (SSE) and shift equivalence (SE), respectively.

\begin{dfn}[Williams]
Let $A$ and $B$ be square matrices with entries in $\mathbb{N}$. $A$ and $B$ are 

\begin{enumerate}
\item \emph{shift equivalent} with lag $m \in \mathbb{N} \setminus \{0\}$ if there are rectangular matrices $R$ and $S$ with entries in $\mathbb{N}$ such that
\begin{align*}
A^m = RS, \ \ B^m = SR, \\
BS = SA, \ \ AR = RB.
\end{align*}

\item \emph{elementary shift related} if they are shift equivalent with lag $1$.

\item \emph{strong shift equivalent} if they are equivalent in the transitive closure of elementary shift relation. 
\end{enumerate}
\end{dfn}

For the purpose of characterizing these shift equivalence relations in terms of Cuntz-Krieger graph C*-algebras, it was noticed early on that the C*-algebras need to be considered together with additional structure involving the canonical gauge action and the canonical diagonal subalgebra \cite{BK00, CR17}. In fact, by a theorem of Krieger \cite{Kri80} we know that the dimension group triples of SFTs are isomorphic if and only if the associated matrices are shift equivalent. Through the use of $K$-theory, this leads to the following corollary.

\begin{cor}[Krieger] \label{c:krieg}
Let $A$ and $B$ be two finite \emph{essential} matrices. If the Cuntz-Krieger graph C*-algebras $\mathcal{O}_A$ and $\mathcal{O}_B$ are stably isomorphic in a way preserving their gauge actions $\gamma^A$ and $\gamma^B$, then $A$ and $B$ are shift equivalent.
\end{cor}

In the 2000s, the works on Cuntz-Krieger graph C*-algebras inspired a systematic study of their purely algebraic counterparts \cite{AGDP04, AAP05, AMP07} known as \emph{Leavitt path algebras}. A phenomenon of great interest in the study of Leavitt path algebras is that many of the results about them seem to mirror corresponding results for Cuntz-Krieger graph C*-algebras, with very different proof techniques. Simplicity, pure infiniteness, primitivity, and $K_0$-group are examples of such similarly behaving properties and invariants. However, there are still some prominent result for Cuntz-Krieger graph C*-algebras that do not lend themselves available in the purely algebraic setting of Leavitt path algebras (see for instance \cite{BK00, ERRS21}).

For the purpose of discovering new invariants for conjugacy of SFTs, it became important to understand to what extent the converse of Corollary \ref{c:krieg} holds. In the context of Leavitt path algebras, Hazrat \cite{Haz-JA13, Haz13} made an analogous conjecture to the converse of Corollary~\ref{c:krieg}, and these two classification problems have become central in this line of research.

\begin{conj}[Hazrat] \label{c:bk}
Let $A$ and $B$ be two finite essential matrices. The following are equivalent:
\begin{enumerate}
\item[(1)] $A$ and $B$ are SE; 

\item[(2)] the Cuntz-Krieger graph C*-algebras $\mathcal{O}_A$ and $\mathcal{O}_B$ are stably isomorphic in a way preserving their gauge actions $\gamma^A$ and $\gamma^B$; 

\item[(3)] the Leavitt path algebras $L_A$ and $L_B$ are graded Morita equivalent.
\end{enumerate}
\end{conj} 

A partial converse to Corollary \ref{c:krieg} was obtained by Bratteli and Kishimoto for primitive matrices \cite{BK00} by proving a non-commutative, topological analogue of Rokhlin property from ergodic theory. This is not surprising since such Rokhlin properties have been used for obtaining classification results of single automorphisms of various classes of C*-algebras \cite{HJ82, HO84, Kis95, Kis96, EK97}. However, in the purely algebraic setting, Rokhlin property and its consequences are not expected to hold, and an entirely new approach is required.

In this paper, we use the theory of bicategories developed by Meyer et. al. \cite{BMZ13, AM16}, as well as the Cuntz-Pimsner homomorphism $\mathcal{O}^0$ developed by Meyer and Sehnem \cite{MS19}, to show that shift equivalence of two essential matrices $A$ and $B$ coincides with equivariant homotopy equivalence of their stabilized Cuntz-Krieger graph C*-algebras $\mathcal{O}_A \otimes \mathbb{K}$ and $\mathcal{O}_B \otimes \mathbb{K}$. More precisely, let $A$ and $B$ be square matrices with entries in $\mathbb{N}$ indexed by $V$ and $W$ respectively. We will say that $\mathcal{O}_A$ and $\mathcal{O}_B$ are stably equivariantly \emph{homotopy equivalent} if there are gauge-preserving $*$-homomorphisms $\tau : \mathcal{O}_A \otimes \mathbb{K} \rightarrow \mathcal{O}_B \otimes \mathbb{K}$ and $\rho : \mathcal{O}_B \otimes \mathbb{K} \rightarrow \mathcal{O}_A \otimes \mathbb{K}$ such that $\tau \circ \rho$ and $\rho \circ \tau$ are homotopic to the respective identity automorphisms in a way preserving the gauge actions.

\begin{theoremintro}\label{thm_intro_SE}
Let $A$ and $B$ be essential square matrices with entries in $\mathbb{N}$. Then $A$ and $B$ are shift equivalent if and only if $\mathcal{O}_A$ and $\mathcal{O}_B$ are stably equivariantly homotopy equivalent.
\end{theoremintro}

Homotopy equivalence in the context of Leavitt path algebras was first considered by Corti\~nas and Montero \cite{CM20, CM23} as an intermediary step for achieving classification up to isomorphism, analogously to known results for Cuntz-Krieger graph C*-algebras (see \cite{ERRS21}). Recently, Arnone \cite{Arn+} was able to prove a formidable homotopy analogue of Hazrat's conjecture under the assumption that $A$ and $B$ are primitive. Although this result is an important intermediary step in resolving the algebraic Hazrat conjecture in the case of primitive matrices, the definition of homotopy equivalence in the purely algebraic setting is complicated by the lack of an underlying topology on general Leavitt path algebras. Moreover, the proofs \cite{CM20, CM23} and \cite{BK00, Arn+} all rely on deep $K$-theory classification techniques, which we are able to circumvent through the use of a direct bicategorical approach.

In general, homotopy is too weak of a notion to conclude isomorphism. However, for simple Cuntz-Krieger graph C*-algebras the Kirchberg-Phillips classification of purely infinite simple C*-algebras \cite{NCP00} assures us that, when the gauge action is ignored, homotopy equivalence implies the existence of a $*$-isomorphism. We are hopeful that with the novelty of avoiding such deep classification theorems, and by combining ideas from bicategory theory and homotopy theory, our techniques will prove useful in resolving Hazrat's Conjecture in full.

This paper contains six sections, including this introductory section. In Section \ref{sec_prelim} we provide some of the necessary basic theory to be used in our paper along with relevant references. In Section \ref{s:bicat-se} we discuss Meyer's bicategory of C*-correspondences, as well as the Meyer--Sehnem Cuntz--Pimsner homomorphism, showcasing the advantage of our approach. In Section \ref{s:homot-se} we define homotopy shift equivalence in the bicategory of C*-correspondences. This is where we show that shift equivalence implies that associated graph C*-correspondences are homotopy shift equivalent, and that homotopy shift equivalence is preserved under the Meyer--Sehnem Cuntz--Pimsner homomorphism. In Section \ref{s:cor-hom} we explain how we may consider certain C*-correspondences as $*$-homomorphisms, where tensor product corresponds to composition. We analyze this behavior under taking crossed products, and show that taking homotopy equivalence is preserved under taking crossed products in an appropriate sense. Finally, in Section \ref{s:main} we prove our main result, Theorem \ref{thm_intro_SE}, and provide sufficient conditions for the existence of an equivariant $*$-isomorphism between Cuntz-Krieger graph C*-algebras.
 
\section{Preliminary}\label{sec_prelim}

In this section we recall some of the necessary theory. We recommend \cite{Raeburn-graph} for the basic theory of graph C*-algebras, and \cite{Lan95} for the theory of C*-correspondences. 

Let $\A$ be a C*-algebra. A right Hilbert $\A$-module $X$ is a linear space with a right action of the C*-algebra $\A$ and an $\A$-valued inner product $\langle \cdot, \cdot \rangle$ such that $X$ is complete with respect to the induced norm $\| \xi \|:= \| \langle \xi,\xi \rangle \|^{1/2}$ for $\xi \in X$.

For a Hilbert $\A$-module $X$ we denote by $\L(X)$ the C*-algebra of all adjointable operators on $X$. For $\xi,\eta \in X$, the rank-one operator $\theta_{\xi,\eta} \in \L(X)$ is defined by $\theta_{\xi,\eta}(\zeta) = \xi \langle \eta,\zeta \rangle$ for $\zeta \in X$. Then, the ideal of \emph{generalized compact operators} $\K(X)$ of $\L(X)$ is the closed linear span of elements of the form $\theta_{\xi,\eta}$ for $\xi,\eta \in X$.

Now let $\B$ be another C*-algebra. An \emph{$\A-\B$ correspondence} is a right Hilbert $\B$-module $X$ together with a $*$-representation $\phi_X : \A \rightarrow \L(X)$.

\begin{asmp}
We assume throughout the paper that correspondences are non-degenerate in the sense that $\phi_X(\A)X = X$. 
\end{asmp}

When we have an $\A-\B$ correspondence, we will often treat it as an $\A-\B$ bimodule. We say the $\A-\B$ correspondence $X$ is
\begin{enumerate}
\item \emph{full} if $\overline{\mathrm{span}} \{ \langle \xi, \eta \rangle : \xi, \eta \in X \} = \B$;
\item \emph{injective} if $\phi_X$ is an injection;

\item \emph{proper} if $\phi_X$ has image in $\K(X)$; and

\item \emph{regular} if it is both injective and proper.

\end{enumerate}

We will say that two $\A-\B$ correspondences $X$ and $Y$ are \emph{unitarily isomorphic} if there is an isometric bimodule surjection $\tau: X \rightarrow Y$, and call $\tau$ a \emph{unitary correspondence isomorphism}. Due to the polarization identity, a surjective bimodule map $\tau$ is a correspondence isomorphism if and only if for $\xi,\eta \in X$ we have
$$
\langle \tau(\xi),\tau(\eta) \rangle = \langle \xi,\eta \rangle.
$$ 

\begin{dfn}
Let $X$ be an $\A-\A$ correspondence. A representation of $X$ on a Hilbert space $\H$ is a pair $(\pi, t)$ where $\pi \colon \A\to \mathbb{B}(\H)$ is a non-degenerate $*$-homomorphism and $t\colon X\to \mathbb{B}(\H)$ is a bounded linear operator such that
\begin{enumerate}
\item $t(\xi)^*t(\eta) = \pi(\langle \xi, \eta\rangle)$;
\item $\pi(a)t(\xi) = t(a \cdot \xi)$
\end{enumerate}
for all $\xi,\eta\in X$ and $a\in \A$. 
\end{dfn}

It follows automatically that $t(\xi)\pi(a) = t(\xi a)$ for all $\xi \in X$ and $a\in \A$. The \emph{Toeplitz algebra} $\T_X$ of $X$ is the universal C$^*$-algebra with respect to representations of $X$. The algebra $\T_X$ possess a canonical point-norm continuous gauge action $\alpha : \bT \rightarrow \Aut(\T_X)$ which fixes $A$ and acts by multiplication by a unimodular scalar on $X$. More precisely, let $(\pi,t)$ be a representation such hat $\T_X \cong C^*(\pi,t)$, then $\alpha_z(t(\xi)) = z\cdot t(\xi)$ and $\alpha_z(\pi(a)) = \pi(a)$ for all $z\in \bT$, $a\in \A$ and $\xi\in X$.

When we are given a representation $(\pi,t)$ of $X$, we have an induced $*$-homomorphism $\psi\colon \K(X)\hookrightarrow \T_X$ which is determined by the property $\psi(\theta_{\xi, \eta}) = \xi\otimes \eta^*$. For our purposes, it will be important to consider Cuntz-Pimsner C*-algebras when the involved C*-correspondence $X$ is \emph{regular}. In this case, the \emph{Cuntz-Pimsner algebra} $\O_X$ of $X$ is the quotient of $\T_X$ by the ideal generated by differences $a-\psi(\varphi_X(a))$ for $a\in A$. Representations of $X$ whose induced $*$-representation defined on $\T_X$ factors through $\O_X$ are called \emph{covariant}. Moreover, since the ideal generated by $a-\psi(\varphi_X(a))$ for $a\in A$ is gauge invariant, we obtain an induced gauge action $\gamma$ of $\T_X$ on $\mathcal{O}_X$ which fixes $\mathcal{A}$ and acts by multiplication by the unimodular scalar on $X$. This action gives rise to a topological $\mathbb{Z}$-grading of $\mathcal{O}_X$ with graded subspaces given by
$$
\mathcal{O}_X^n := \{ \ T \in \mathcal{O}_X \ | \ \gamma_z(T) = z^n\cdot T \ \}.
$$

For an $\A-\B$ correspondence $X$ and a $\B-\C$ correspondence $Y$, we may define the interior tensor product $\A -\C$ correspondence $X \otimes_{\B} Y$ as follows. Let $X \odot_{\B} Y$ be the quotient of the algebraic tensor product by the subspace generated by elements of the form
$$
\xi b \otimes \eta - \xi \otimes b \eta, \ \ \text{for} \ \ \xi \in X, \ \eta \in Y, \ b \in \B.
$$
We define a $\C$-valued inner product and a left $\A$ action by setting
$$
\langle \xi \otimes \eta, \xi' \otimes \eta' \rangle = \langle \eta , \langle \xi ,\xi' \rangle \eta' \rangle, \ \ \text{for} \ \ \xi,\xi' \in X, \ \ \eta,\eta' \in Y
$$
$$
a \cdot (\xi \otimes \eta) = (a\xi) \otimes \eta, \ \ \text{for} \ \ \xi \in X, \ \ \eta \in Y, \ \ a \in \A,
$$
and denote by $X \otimes_{\B} Y$ the separated completion of $X \odot_{\B} Y$ with respect to the $\C$-valued semi-inner product above. It then follows that $X\otimes_{\B} Y$ is an $\A-\C$ correspondence. 

\begin{dfn}
Let $\A$ and $\B$ be C*-algebras.  A \emph{Hilbert $\A-\B$ bimodule} $X$ is a complex vector space which is both a right Hilbert $\B$-module and a left Hilbert $\A$-module, such that 
\begin{enumerate}
\item The left action of $\B$ is by adjointables on the right Hilbert $\A$-module structure, and the right action of $\A$ is by adjointables on the left Hilbert $\B$-module structure.

\item for all $x, y, z \in X$, 
$$
_\A \langle x, y \rangle \cdot z = x \cdot \langle y, z \rangle_\B.
$$
\end{enumerate}
If additionally we have that the right and left Hilbert C*-module structures on $X$ are full, we will say that $X$ is an \emph{imprimitivity $\A-\B$ bimodule}
\end{dfn}

The following lemma connects $\A-\B$ correspondences and imprimitivity $\A-\B$-bimodules, and is a well-known consequence of \cite[Proposition~3.8]{RW-Morita-Eq}.

\begin{lem}\label{lem-iso-impri}
Let $X$ be a $\A-\B$-correspondence such that $X$ is a full right Hilbert $B$-module.  Assume the left action $\phi_{X} \colon \A \to \L(X)$ is injective such that $\phi_X(\A)=\K(X)$.  Then $X$ is an $\A-\B$ imprimitivity bimodule with left inner product given by $_\A \langle \xi, \xi' \rangle = \phi_X^{-1}(\theta_{\xi, \xi'})$.
\end{lem}

\begin{proof}
By \cite[Proposition~3.8]{RW-Morita-Eq}, $X$ is a $\K(X)-B$ imprimitivity bimodule with left inner product given by $_{\K(X)}\langle \xi , \xi' \rangle= \theta_{\xi, \xi'}$.  Since $\phi_X \colon A \to \K(X)$ is an isomorphism, $_A \langle \xi, \xi' \rangle := \phi_X^{-1} ( \theta_{\xi, \xi' } )$ is a left inner product which makes $X$ an $A-B$ imprimitivity bimodule.
\end{proof}

A directed graph $G=(V,E,s,r)$ is composed of a set of vertices $V$ and a set of edges $E$, along with source and range maps $s,r: E \rightarrow V$. We say that $G$ is finite if $V$ and $E$ are finite. Following \cite{CDE24}, for a $V \times W$ matrix $F = [F_{ij}]$ with entries in $\mathbb{N}$, we denote 
$$
E_F := \{ (v, \alpha, w) \ \mid \  0 \leq \alpha < F_{vw}, v \in V, w \in W\}
$$ 
so that $r(v, \alpha, w) = w$ and $s(v, \alpha, w) = v$, and $\alpha \in \mathbb{N}$. When $V = W$, this makes $G_F := (V, E_F, r, s )$ into a directed graph in its own right. 

One of our main examples of C*-correspondences arise from matrices as above. Let $V$ and $W$ be finite, and denote by $\A = c(V)$ and $\B = c(W)$ be the finite dimensional abelian C*-algebras consisting of functions on $V$ and $W$ respectively. Let $R$ be a $V\times W$ matrix with entries in $\bbN$. We define an $\A-\B$ correspondence $X(R)$ from $E_R$ as follows.
Let $X(R)$ be the space of functions from $E_R$ to $\bbC$.
The bimodule structure on $X(R)$ is defined by
\[
  (a\cdot \xi \cdot b)(e) = a(r(e))\xi(e) b(s(e))
\]
for all $a\in \A$, $b\in \B$, $\xi \in X(R)$, and $e\in E_R$.
The $\B$-valued inner product on $X(R)$ is given by
\[
  \langle \xi, \eta\rangle(w) = \sum_{s(e)=w}\overline{\xi(e)}\eta(e)
\]
for all $\xi,\eta\in X(R)$ and $w\in W$. With this $c(W)$-valued inner product and $c(V)$ left action, $X(R)$ becomes a (finite-dimensional) $\A-\B$ correspondence.
When the matrix $R$ is essential, the correspondence $X(R)$ is both full and regular. If $A$ is a $V\times V$ matrix with entries in $\bbN$, we say that $X(A)$ is the \emph{graph correspondence} of $G_A$. 

In this paper we will conduct our study through the use of graph C*-algebras, which include the class of Cuntz-Krieger C*-algebras (where the adjacency matrices can only have values in $\{0,1\}$). This class of C*-algebras forms a robust class of Cuntz-Pimsner algebras. For more on the subject, we recommend \cite{FLR00, Raeburn-graph}. 

\begin{dfn}
Let $G = (V, E, s,r)$ be a finite directed graph, which we assume has no sinks and no sources. A Cuntz-Krieger $G$-family $\{P_v,S_e \}_{v\in V, e\in E}$ on a Hilbert space $\H$ consists of mutually orthogonal projections $\{P_v\}_{v\in V}$ and partial isometries $\{S_e\}_{e\in E}$ on $\H$ such that
\begin{enumerate}
\item $S_e^*S_e = P_{s(e)}$ for all edges $e\in E$;
\item $P_v = \sum_{e\in E, s(e)=v} S_eS_e^*$.
\end{enumerate}
The \emph{the graph C*-algebra} of $G$ is the C*-algebra generated by a universal Cuntz-Krieger $G$-family.
\end{dfn}

More precisely, a Cuntz-Krieger family $\{Q_v,T_e\}_{v\in V, \ e\in E}$ is universal if for any Cuntz-Krieger family $\{P_v,S_e \}_{v\in V, e\in E}$ there is a surjective $*$-homomorphism $C^*(\{Q_v,T_e\}) \rightarrow C^*(\{P_v,S_e\})$ satisfying $Q_v \mapsto P_v$ and $T_e \mapsto S_e$. We then denote by $C^*(G):= C^*(\{Q_v,T_e\})$ the graph C*-algebra of $G$, which is uniquely defined up to generator-preserving $*$-isomorphism.

It turns out, that graph C*-algebras are exactly the Cuntz-Pimsner algebras of the \emph{graph correspondence} of the adjacency matrix of the graph. That is, $C^*(G) \cong \mathcal{O}_{X(A_G)}$ where $A_G$ is the adjacency matrix of $G$. Moreover, through this identification the usual gauge action $\gamma$ on $\mathcal{O}_{X(A_G)}$ is given on generators by $\gamma_z(Q_v) = Q_v$ and $\gamma_z(T_e) = z \cdot T_e$ for $v\in V$ and $e\in E$.

\section{A bicategorical approach to shift equivalence} \label{s:bicat-se}

We introduce the C*-bimodule analog of aligned equivalence of matrices, which appeared first as the conclusion of \cite[Lemma 4.2]{CDE24}. The purely algebraic, bimodule version of aligned shift equivalence was extensively studied in \cite{ART23}. We will interpret aligned shift equivalence of C*-bimodules using Meyer's bicategory of C*-correspondences (see \cite[Subsection 2.2]{BMZ13} and \cite[Appendix A]{MS19}). By \cite[Proposition~3.5]{CDE24}, two matrices $A$ and $B$ with entries in $\mathbb{N}$ are shift equivalent with lag $m$ if and only if $X(A)$ and $X(B)$ are shift equivalent with lag $m$ in the sense that there exists a concrete shift between $X(A)$ and $X(B)$ with lag $m$ as in the following definition.

\begin{dfn}\label{dfn_mase}
Let $X$ be an $\mathcal{A}$-correspondence and $Y$ be a $\mathcal{B}$-correspondence. We say that the tuple $(M, N, \Phi_M, \Phi_N, \Psi_X, \Psi_Y)$ is a \emph{concrete shift between $X$ and $Y$ with lag $m \in \mathbb{N}\setminus \{0\}$} if $M$ is an $\A-\B$-correspondence, $N$ is a $\B-\A$-correspondence, and each map 
\begin{align*}
\Phi_M &\colon X \otimes_{\mathcal{A} } M  \to M \otimes_{\mathcal{B}} Y  & \Phi_N &\colon Y \otimes_{\mathcal{B} } N  \to N \otimes_{\mathcal{A}} X \\
\Psi_X &\colon M \otimes_{\mathcal{B}} N \to X^{\otimes m} & \Psi_Y &\colon N \otimes_{\mathcal{A}} M \to Y^{\otimes m} 
\end{align*}
are unitary correspondence isomorphisms. We say $(M, N, \Phi_M, \Phi_N, \Psi_X, \Psi_Y)$ is \emph{aligned} if additionally
\begin{align*}
     (\Psi_X \otimes \id_X)(\id_M \otimes \Phi_N)(\Phi_M \otimes \id_N) &= (\id_X \otimes \Psi_X) \\
    (\Psi_Y \otimes \id_Y)(\id_N \otimes \Phi_M)(\Phi_N \otimes \id_M) &= (\id_Y \otimes \Psi_Y).
\end{align*}
We say that $X$ and $Y$ are \emph{aligned} shift related with $\lag$ $m$ if there exists a concrete shift between $X$ and $Y$ with $\lag$ $m$ which is aligned.
\end{dfn}

Note that we are suppressing the associativity isomorphism $(X_1 \otimes_\C X_2) \otimes_{\D} X_3 \cong X_1 \otimes_\C (X_2 \otimes_\D X_3)$ given by $(x_1 \otimes x_2) \otimes x_3 \mapsto x_1 \otimes (x_2 \otimes x_3)$ to simplify the notation.

It turns out that the bicategorical approach developed by Meyer is especially useful for interpreting aligned shift equivalence of C*-correspondences. We define the structure of a bicategory $\mathfrak{C}^{\mathbb{N}}_{\mathrm{pr}}$ of proper and full C*-correspondences (see \cite[Definition 3.1]{MS19}).

\begin{enumerate}
\item The objects of this category are pairs $(\A,X)$ where $\A$ is a C*-algebra and $X$ is a full and regular C*-correspondence over $\A$.

\item Let $(\A,X)$ and $(\B,X)$ be objects. A $1$-arrow from $(\A,X)$ to $(\B,Y)$ is a pair $[F,\Phi_F] : (\B,Y) \leftarrow (\A,X)$ where $F$ is a full and proper $\B-\A$ correspondence and $\Phi_F : Y \otimes_{\B} F \rightarrow F \otimes_{\A} X$ is a unitary isomorphism of $\B-\A$ correspondences. We say that a $1$-arrow $[F,\Phi_F]$ is an equivalence arrow if $F$ is an imprimitivity bimodule.

\item Let $[F,\Phi_F], [G,\Phi_G] : (\B,Y) \leftarrow (\A,X)$ be $1$-arrows between the same pair of objects. A $2$-arrow $\Psi: [F,\Phi_F] \rightarrow [G,\Phi_G]$ between $1$-arrows $[F,\Phi_F]$ and $[G,\Phi_G]$ is a unitary correspondence isomorphism $\Psi : F \rightarrow G$ which makes the following diagram commute
\begin{equation}\label{eq:2-arrow}
\begin{tikzcd}
                        {Y\otimes_\B F} & {F\otimes_\A X}\\
                        {Y\otimes_\B G} & {G\otimes_\A X}
                        \arrow["{1_Y\otimes \Psi}", rightarrow, from=1-1, to=2-1]
                        \arrow["{\Phi_F}", rightarrow, from=1-1, to=1-2]
                        \arrow["{\Phi_G}"', rightarrow, from=2-1, to=2-2]
                        \arrow["{\Psi \otimes 1_X}", rightarrow, from=1-2, to=2-2].
\end{tikzcd}
\end{equation}
\end{enumerate}

When $(\A,X), (\B,Y), (\C,Z)$ are objects, and $[G,\Phi_G] : (\C,Z) \leftarrow (\B,Y)$ and $[F,\Phi_F] : (\B,Y) \leftarrow (\A,X)$ are $1$-arrows, the composition $1$-arrow $[G,\Phi_G] \otimes_{\B}[F,\Phi_F]$ is given by\- $[G\otimes_{\B}F, \Phi_G \odot \Phi_F] : (\C,Z) \leftarrow (\A,X)$, where $\Phi_G \odot \Phi_F := (1_G \otimes \Phi_F) \circ (\Phi_G \otimes 1_F)$. When $[F,\Phi_F], [G,\Phi_G], [H,\Phi_H] : (\B,Y) \leftarrow (\A,X)$ are $1$-arrows with $2$-arrows $\Psi_1 : [F,\Phi_F] \rightarrow [G,\Phi_G]$ and $\Psi_2 : [G,\Phi_G] \rightarrow [H,\Phi_H]$, the \emph{vertical} composition $\Psi_2 \circ \Psi_1 : [F,\Phi_F] \rightarrow [H,\Phi_H]$ is also a $2$-arrow. Additionally, if there are 2-arrows $\Psi_i\colon [F_i, \Phi_{F_i}] \to [G_i, \Phi_{G_i}]$ for $i=1,2$, then the \emph{horizontal} composition is the arrow $\Psi_1\otimes \Psi_2 \colon [F_1\otimes G_1, \Psi_{F_1}\odot \Psi_{G_1}] \to [F_2\otimes G_2, \Psi_{F_2}\odot \Psi_{G_2}]$.

An important property of the bicategory $\mathfrak{C}^{\mathbb{N}}_{\mathrm{pr}}$ is that all 2-arrows are invertible. Indeed, let $\Psi \colon [F, \Phi_F] \to [G, \Phi_G]$ be a 2-arrow which is a unitary correspondence isomorphism $\Psi \colon F\to G$ satisfying \eqref{eq:2-arrow}. Then, $\Psi^{-1} \colon G\to F$ is also a unitary isomorphism satisfying \eqref{eq:2-arrow}. Therefore, two $1$-arrows are isomorphic if and only if there is a $2$-arrow between them.

Every object $(\A,X)$ has the natural self $1$-arrow $[\mathcal{A}, 1_X]$, where we write $1_X$ to denote the natural map $X \otimes_{\A} \A \rightarrow \A \otimes_\A X$.  It is easy to verify directly that $[\A,1_X]$ acts as a unit for composition of $1$-arrows, and that composition of $1$-arrows is associative. Moreover, there is a canonical $1$-arrow $[X,1_{X\otimes_{\A}X}]$ which induces a sequence of $1$-arrows $[X^{\otimes m}, 1_{X^{\otimes m+1}}]$, where we have abused notation and identified the associativity isomorphism $X\otimes_{\A} X^{\otimes m} \cong X^{\otimes m} \otimes X$ with $1_{X^{\otimes m+1}}$. 

Suppose now that $(M, N, \Phi_M, \Phi_N, \Psi_X, \Psi_Y)$ is a concrete shift between $(\A,X)$ and $(\B,Y)$ of lag $m$. The key observation now is that the concrete shift $(M, N, \Phi_M, \Phi_N, \Psi_X, \Psi_Y)$ is \emph{aligned} exactly when the unitary isomorphisms $\Psi_X$ and $\Psi_Y$ can be chosen to be $2$-arrows $\Psi_X: [M\otimes_{\B} N, \Phi_M \odot \Phi_N] \rightarrow [X^{\otimes m}, 1_{X^{\otimes m+1}}]$ and $\Psi_Y: [N\otimes_{\A} M, \Phi_M \odot \Phi_N] \rightarrow [Y^{\otimes m}, 1_{Y^{\otimes m+1}}]$. This realization of aligned shift equivalence as $2$-arrows in the bicategory $\mathfrak{C}_{\mathrm{pr}}^{\mathbb{N}}$ will allow us to prove transitivity of aligned shift relation of C*-correspondences from abstract principles.

\begin{lem}\label{lem:ind-iso-mor}
Let $[F,\Phi_F]: (\B,Y) \leftarrow (\A,X)$ be a $1$-arrow. Then $\Phi_F$ induces a $2$-arrow $[Y,1_{Y\otimes Y}] \otimes [F,\Phi_F] \to [F,\Phi_F] \otimes [X,1_{X\otimes X}]$. Consequently, the $1$-arrows $[F,\Phi_F] \otimes [X^{\otimes m},1_{X^{\otimes m+1}}]$ and $[Y^{\otimes m},1_{Y^{\otimes m+1}}]\otimes[F,\Phi_F]$ are isomorphic.
\end{lem}

\begin{proof}
By definition, $\Phi_F :Y \otimes_{\B} F \rightarrow F\otimes_{\A} X$ is a unitary correspondence isomorphism, and comprises a 2-arrow if it satisfies \eqref{eq:2-arrow}. First, we show that the following diagram commutes.
\[
\begin{tikzcd}
                  {Y\otimes (Y\otimes F)} & {(Y\otimes F)\otimes X}\\
                  {Y\otimes(F\otimes X)} & {(F\otimes X)\otimes X}
                  \arrow["{\Phi_F\odot 1_{X^{\otimes2}}}"', rightarrow, from=2-1, to=2-2]
                  \arrow["{\Phi_F\otimes 1_X}", rightarrow, from=1-2, to=2-2]
                  \arrow["{1_{Y^{\otimes2}}\odot \Phi_F}", rightarrow, from=1-1, to=1-2]
                  \arrow["{1_Y\otimes\Phi_F}"', rightarrow, from=1-1, to=2-1].
\end{tikzcd}
\]
By definition, we have $1_{Y^{\otimes 2}} \odot \Phi_F = (1_{Y} \otimes \Phi_F)\circ (1_{Y^{\otimes 2}} \otimes 1_F)= 1_Y\otimes \Phi_F$ and similarly $\Phi_F\odot 1_{X^{\otimes 2}} = \Phi_F \otimes 1_X$.
Thus, the diagram above is the same as the following diagram which clearly commutes
\[
\begin{tikzcd}
                  {Y\otimes (Y\otimes F)} & {(Y\otimes F)\otimes X}\\
                  {Y\otimes(F\otimes X)} & {(F\otimes X)\otimes X}
                  \arrow["{\Phi_F\otimes 1_X}"', rightarrow, from=2-1, to=2-2]
                  \arrow["{\Phi_F\otimes 1_X}", rightarrow, from=1-2, to=2-2]
                  \arrow["{1_Y\otimes\Phi_F}", rightarrow, from=1-1, to=1-2]
                  \arrow["{1_Y\otimes\Phi_F}"', rightarrow, from=1-1, to=2-1].
\end{tikzcd}.
\]
Therefore, the 1-arrows $[F,\Psi_F]\otimes [X,1_{X^{\otimes 2}}]$ and $[Y,1_{Y^{\otimes 2}}]\otimes [F, \Psi_F]$ are isomorphic via a $2$-arrow. Lastly, the second part of the statement follows by induction.
\end{proof}

To illustrate the advantage of using the language of bicategories, we prove that ASE of C*-correspondences is an equivalence relation.

\begin{prp} \label{prp:transitivity-mod-asr}
Aligned shift relation between C*-correspondences which are full and regular is an equivalence relation.
\end{prp}

\begin{proof}
Symmetry and reflexivity are clear. Thus, we need only prove transitivity. Thus, suppose that $(\A,X),(\B,Y),(\C,Z)$ are objects, and assume that $(M_1, N_1, \Phi_{M_1}, \Phi_{N_1}, \Psi_X, \Psi_Y)$ and $(M_2, N_2, \Phi_{M_2}, \Phi_{N_2}, \Psi_Y, \Psi_Z)$ are concrete aligned shifts of lags $m$ and $n$ respectively. We want to show that the $1$-arrows $[M,\Phi_M] = [M_1, \Phi_{M_1}] \otimes [M_2, \Phi_{M_2}]$ and $[N,\Phi_N] = [N_2, \Phi_{N_1}] \otimes [N_1, \Phi_{N_1}]$ give rise to $2$-arrows $[M,\Phi_M] \otimes [N, \Phi_N] \cong [X^{\otimes n+m},1_{X^{\otimes n+m+1}}]$ and $[N,\Phi_N] \otimes [M,\Phi_M] \cong [Z^{\otimes n+m}, 1_{Z^{\otimes n+m+1}}]$. We will show the $2$-arrow for $X$, and the one for $Z$ will follow similarly.

By definition and Lemma \ref{lem:ind-iso-mor} we have a unitary correspondence isomorphism $\Psi_X'$ defined by the following chain of $2$-arrows
\[
[M,\Phi_M] \otimes [N, \Phi_N] \cong [M_1,\Phi_{M_1}] \otimes ([M_2,\Phi_{M_2}] \otimes [N_2,\Phi_{N_2}]) \otimes [N_1,\Phi_{N_1}] \cong 
\]
\[
 [M_1,\Phi_{M_1}] \otimes [Y^{\otimes n},1_{Y^{\otimes n+1}}] \otimes [N_1,\Phi_{N_1}] \cong [M_1,\Phi_{M_1}] \otimes [N_1,\Phi_{N_1}] \otimes [X^{\otimes n},1_{X^{\otimes n+1}}] \cong
\]
\[
[X^{\otimes m},1_{X^{\otimes m+1}}] \otimes [X^{\otimes n},1_{X^{\otimes n+1}}] \cong [X^{\otimes m+n},1_{X^{\otimes m+n+1}}].
\]
Similarly, we define $\Psi_Z'$ with the appropriate analogous chain of identities. Since composition of $2$-arrows is again a $2$-arrow, we see that $(M,N,\Phi_M,\Phi_N,\Psi_X',\Psi_Z')$ comprises a concrete aligned shift of lag $m+n$, and we see that $(\A,X)$ and $(\C,Z)$ are aligned shift related.
\end{proof}

In the work of Meyer and Sehnem \cite{MS19}, a homomorphism (which is the analogue of a functor in the context of bicategories) from $\mathfrak{C}^{\mathbb{N}}_{\mathrm{pr}}$ to the sub-bicategory of imprimitivity bimodules was constructed. Denote by $\mathfrak{C}_{\mathrm{pr},*}^{\mathbb{N}}$ the full sub-bicategory of $\mathfrak{C}^{\mathbb{N}}_{\mathrm{pr}}$ with objects $(\A, X)$ such that $X$ is an imprimitivity bimodule. Next, we describe a homomorphism from $\mathfrak{C}^{\mathbb{N}}_{\mathrm{pr}}$ to $\mathfrak{C}_{\mathrm{pr},*}^{\mathbb{N}}$.

First, an object $(\A,X)$ is mapped to the pair $(\O_X^0, \O_X^1)$, where $X\otimes_{\A} \O_X^0$ is identified as the graded component $\O_X^1$ of $1\in \mathbb{Z}$ of the Cuntz-Pimsner C*-algebra. Now, suppose that $[F,\Phi_F] : (\A,X) \leftarrow (\B,Y)$ is a $1$-arrow. Consider the $\A -\O_Y^0$ correspondence given by $\O_{F,\Phi_F}^0:= F\otimes_{\B} \O_Y^0$. Then, the map $\Phi_F$ induces a unitary isomorphism of correspondences by
\[
\widehat{\Phi}_F \colon X\otimes_{\A} \O_{F, \psi_F}^0 \xrightarrow{\Phi_F\otimes 1_{\O_Y^0}} F\otimes_{\B} Y\otimes_{\B} \O_Y^0 \cong F\otimes_{\B} \O_Y^1 \cong \O_{F, \Phi_F}^0 \otimes_{\O_Y^0} \O_Y^1.
\]
Hence, the pair $[\O_{F, \psi_F}^0,\widehat{\Phi}_F]$ is a $1$-arrow from $(\O_Y^0, \O_Y^1)$ to $(\A,X)$. By \cite[Proposition 3.4]{MS19} it extends \emph{uniquely} to a $1$-arrow from $(\O_Y^0, \O_Y^1)$ to $(\O_X^0, \O_X^1)$ which we continue to denote by $[\O_{F, \psi_F}^0,\widehat{\Phi}_F]$. By \cite[Corollary 4.7]{MS19} the above assignment of objects $(\A,X) \mapsto (\O_X^0,\O_X^1)$ and of $1$-arrows $[F,\Phi_F] \mapsto [\O_{F,\Phi_F}^0, \widehat{\Phi}_F]$ defines a homomorphism from the bicategory $\mathfrak{C}^{\mathbb{N}}_{\mathrm{pr}}$ to the bicategory $\mathfrak{C}^{\mathbb{N}}_{\mathrm{pr},*}$. In particular, if $[F,\Phi_F] : (\A,X) \leftarrow (\B,Y)$ and $[G,\Phi_F] : (\B,Y) \leftarrow (\C,Z)$ are $1$-arrows, then there is a canonical $2$-arrow 
\begin{equation} \label{eq:canonical-iso-bicat-imp}
[\O_{F,\Phi_F}^0, \widehat{\Phi}_F] \otimes [\O_{G,\Phi_G}^0, \widehat{\Phi}_G] \cong [\O_{F\otimes G, \Phi_F \odot \Phi_G}^0, \widehat{\Phi}_{F\otimes G}].
\end{equation}

\begin{rmk}
We warn the reader that when $[F,\Phi_F] : (\A,X) \leftarrow (\B,Y)$ is a $1$-arrow and $X,Y$ are objects in $\mathfrak{C}^{\mathbb{N}}_{\mathrm{pr}}$, even when $F$ is regular and full (so that $\O_{F, \psi_F}^0$ is regular and full), the $1$-arrow $[\O_{F, \psi_F}^0,\widehat{\Phi}_F]$ may generally fail to be an equivalence $1$-arrow. 

Thus, when $(M,N,\Phi_M,\Phi_N,\Psi_X,\Psi_Y)$ is a concrete shift with lag $m$ between $X$ and $Y$ in $\mathfrak{C}^{\mathbb{N}}_{\mathrm{pr}}$, the C*-correspondences $\O_{M, \psi_M}^0$ and $\O_{N, \psi_N}^0$ may fail to coincide with $M_{\infty}$ and $N_{\infty}$ as defined in \cite{KK-JFA2014} and \cite{CDE24}. This is because $M_{\infty}$ and $N_{\infty}$ are imprimitivity bimodules defined by using the identifications arising from $\Psi_X$ and $\Psi_Y$ (see \cite[Section 6]{CDE24}).
\end{rmk}

We will need the following lemma, which describes the image of $[X^{\otimes m},1_{X^{\otimes m+1}}]$ under the homomorphism defined above. Note that since $(\A,X)$ is an object in $\mathfrak{C}^{\mathbb{N}}_{\mathrm{pr}}$ we have that $(\O_X^1)^{\otimes m} \cong \O_X^m$ as $\O_X^0$-correspondences, and we identify these two imprimitivity bimodules.

\begin{lem} \label{lem:ident-obj-cuntz}
Let $(\A,X)$ be an object in $\mathfrak{C}^{\mathbb{N}}_{\mathrm{pr}}$. Then $[\O^0_{X^{\otimes m},1_{X^{\otimes m+1}}}, \widehat{1}_{X^{\otimes m+1}}]$ is isomorphic to $[\O_X^m,1_{\O_X^{m+1}}]$.
\end{lem}

\begin{proof}
By \cite[Proposition 3.4]{MS19} it suffices to show that there is a $2$-arrow between the $1$-arrows $[\O^0_{X^{\otimes m},1_{X^{\otimes m+1}}}, \widehat{1}_{X^{\otimes m+1}}]$ and $[\O_X^m,1_{\O_X^{m+1}}]$, when merely considered as $1$-arrows from $(\O_X^0, \O_X^m)$ to $(\A,X^{\otimes m})$. Recall that $\O^0_{X^{\otimes m},1_{X^{\otimes m+1}}} := X^{\otimes m}\otimes_{\A} \O_X^0$, so that the map $\Psi : \O_X^m \rightarrow X^{\otimes m}\otimes_{\A} \O_X^0$, which is the inverse of the natural unitary correspondence isomorphism $X^{\otimes m}\otimes_{\A} \O_X^0 \rightarrow \O_X^m$ makes the following diagram commute
\[
\begin{tikzcd}
                        {X\otimes_{\A} \O_X^m} & {\O_X^m \otimes_{\O_X^0} \O_X^1}\\
                        {X\otimes_{\A} \O_{X^{\otimes m},1_{X^{\otimes m+1}}}} & {\O_{X^{\otimes m},1_{X^{\otimes m+1}}}\otimes_{\O_X^0} \O_X^1}.
                        \arrow["{1_X\otimes \Psi}", rightarrow, from=1-1, to=2-1]
                        \arrow["{1_{\O_X^{m+1}}}", rightarrow, from=1-1, to=1-2]
                        \arrow["{\widehat{1}_{X^{\otimes m+1}}}"', rightarrow, from=2-1, to=2-2]
                        \arrow["{\Psi \otimes 1_{\O_X^1}}", rightarrow, from=1-2, to=2-2]
\end{tikzcd}
\]
Hence, we see that $\Psi$ is a $2$-arrow from $[\O_X^m,1_{\O_X^{m+1}}]$ to $[\O^0_{X^{\otimes m},1_{X^{\otimes m+1}}}, \widehat{1}_{X^{\otimes m+1}}]$.
\end{proof}

\section{Homotopy shift equivalence} \label{s:homot-se}

Considering the bicategorical approach in the previous section, it is natural to ask in what sense shift equivalence is functorial. As shown in \cite{CDE24}, merely considering $\mathfrak{C}^{\mathbb{N}}_{\mathrm{pr}}$ as a category is not enough. However, keeping track of identifications (in the sense of $2$-arrows), as is done in a bicategory, provides us with a blueprint that allows for significant results to emerge.

Let $I=[0,1]$ be the unit interval. For a right Hilbert $\B$-module $F$, we denote by $C(I,F)$ the natural right Hilbert $C(I,\B)$-module. If $F$ is an $\A-\B$ correspondence, then $C(I,F)$ is naturally endowed with the structure of a $C(I,\A)-C(I,\B)$ correspondence.

For any $t\in [0,1]$ there is a $C(I,\A) - \A$ correspondence $\A_t$ which is merely $\A$ as a right Hilbert C*-module, together with the left action $f\cdot a = f(t)a$ for all $f\in C(I,\A)$ and $a\in \A$. In what follows, whenever $F$ is a $\B - C(I,\A)$ correspondence, we denote $F_t:= F \otimes_{C(I,\A)} \A_t$ for the evaluation of $F$ at $t\in [0,1]$. On the other direction, when $F$ is a $\B-\A$ correspondence, we can treat $F$ as a $\B - C(I,\A)$ correspondence where multiplication on the left is by constant $\A$-valued functions.

The assignment $A\mapsto C(I,\A)$ and $F \mapsto C(I,F)$ defines a homomorphism from $\mathfrak{C}^{\mathbb{N}}_{\mathrm{pr}}$ to itself. For any $t\in I$ there is a $1$-arrow $\ev_t:= [\A_t,\eta_t] : (C(I,\A),C(I,X)) \leftarrow (\A,X)$, where $\eta_t$ is the canonical unitary correspondence isomorphism $C(I,X) \otimes_{C(I,\A)} \A_t \rightarrow \A_t \otimes_{\A} X$. On the other direction, there is a $1$-arrow $[C(I,X), \eta] : (\A,X) \leftarrow (C(I,\A),C(I,X))$ where multiplication on the left is by constant functions. Whenever $[F,\Phi_F] : (\B,Y) \leftarrow (C(I,\A),C(I,X))$ is a $1$-arrow, we will write $[F_t,\Phi_{F_t}] : (\B,Y) \leftarrow (\A,X)$ for the $1$-arrow defined by $[F,\Phi_F] \otimes [\A_t,\eta_t] = [F_t, \Phi_{F_t} \odot \eta_t]$.

Let $F,G$ be two full and proper $\B-\A$ correspondences. A \emph{homotopy} from $F$ to $G$ is a triple $(H,h_0,h_1) : F \sim G$ where $H$ is a full and proper $\B-C(I,\A)$ correspondence and $h_0: H_0 \rightarrow F$ and $h_1 : H_1 \rightarrow G$ are unitary correspondence isomorphisms. For our purposes this will not be enough, since we need to consider the notion of homotopy for 1-arrows in the bicategory of correspondences.

\begin{dfn}
Let $[F,\Phi_F], [G,\Phi_G] : (\B,Y) \leftarrow (\A,X)$ be $1$-arrows. A \emph{homotopy of $1$-arrows} from $[F,\Phi_F]$ to $[G,\Phi_G]$ is a quadruple $(H,\Phi_H,h_0,h_1)$ consisting of a $1$-arrow $[H,\Phi_H] : (\B,Y) \leftarrow (C(I,\A), C(I,X))$ together with $2$-arrows $h_0 :[H_0,\Phi_{H_0}] \rightarrow [F, \Phi_F]$ and $h_1 : [H_1, \Phi_{H_1}] \rightarrow [G,\Phi_G]$. When there is a homotopy between the $1$-arrows $[F,\Phi_F]$ and $[G,\Phi_G]$, we will say that they are homotopy equivalent, and denote this by $[F,\Phi_F] \sim [G,\Phi_G]$
\end{dfn}

It is straightforward to verify that homotopy of $1$-arrows in $\mathfrak{C}^{\mathbb{N}}_{\mathrm{pr}}$ is an equivalence relation. Now, since aligned shift equivalence of C*-correspondences is phrased in terms of $1$-arrows in the bicategory $\mathfrak{C}^{\mathbb{N}}_{\mathrm{pr}}$, it leads us to the following homotopy version of shift equivalence.

\begin{dfn}
Let $(\A,X)$ and $(\B,Y)$ be objects in $\mathfrak{C}^{\mathbb{N}}_{\mathrm{pr}}$, and let $[M,\Phi_M] : (\A,X) \leftarrow (\B,Y)$ and $[N,\Phi_N]: (\B,Y) \leftarrow (\A,X)$ be $1$-arrows. We say that the quadruple $(M,N,\Phi_M, \Phi_N)$ is a concrete \emph{homotopy shift of lag $m$} between $(\A,X)$ and $(\B,Y)$ if the $1$-arrows $[M\otimes N, \Phi_M \odot \Phi_N]$ and $[N\otimes M, \Phi_N \odot \Phi_M]$ are \emph{homotopic} to $[X^{\otimes m},1_{X^{\otimes m+1}}]$ and $[Y^{\otimes m},1_{Y^{\otimes m+1}}]$ respectively. We say that $(\A,X)$ and $(\B,Y)$ are \emph{homotopy shift equivalent with lag $m$} if there is a concrete homotopy shift of lag $m$ between them.
\end{dfn}

By essentially the same proof as that of Proposition \ref{prp:transitivity-mod-asr}, up to appropriately replacing $2$-arrows with homotopy equivalence of $1$-arrows, it follows that homotopy shift equivalence is an equivalence relation as well. The advantage of considering homotopy shift equivalence is that for graph C*-correspondences it is implied by (and in fact equivalent to) shift equivalence of the underlying adjacency matrices.

\begin{prp} \label{p:se=hse}
Let $A$ and $B$ be finite essential square matrices over $V$ and $W$ respectively, with entries in $\mathbb{N}$. If $A$ and $B$ are shift equivalent of lag $m$ then the C*-correspondences $X(A)$ and $X(B)$ are homotopy shift equivalent of lag $m$.
\end{prp}

\begin{proof}
Suppose that $A$ and $B$ are shift equivalent via $R$ and $S$. We will show that $X(A)$ and $X(B)$ are homotopy shift equivalent. Let $(X(R),X(S),\Phi_{X(R)},\Phi_{X(S)},\Psi_{X(A)},\Psi_{X(B)})$ be the concrete shift of lag $m$ between $X(A)$ and $X(B)$ induced by $R$ and $S$. We need to show that $[X(R)\otimes X(S), \Phi_{X(R)} \odot \Phi_{X(S)}]$ and $[X(S) \otimes X(R),\Phi_{X(S)} \odot \Phi_{X(R)}]$ are homotopic to $[X(A)^{\otimes m}, 1_{X(A)^{\otimes m+1}}]$ and $[X(B)^{\otimes m}, 1_{X(B)^{\otimes m+1}}]$ respectively. 

It will suffice to establish only one of these homotopies of $1$-arrows, and the other will follow similarly. To show that this homotopy of $1$-arrows exists, it will suffice to show that any $1$-arrow $[X(A)^{\otimes m}, \Phi]$ is homotopic to the identity $1$-arrow $[X(A)^{\otimes m}, 1_{X(A)^{\otimes m+1}}]$. As a topological space of unitary correspondence isomorphisms, $\U(X(A)\otimes X(A)^{\otimes m}, X(A)^{\otimes m} \otimes X(A))$ is homeomorphic to a product of unitary groups on finite dimensional inner product space. More precisely,
$$
\U(X(A)\otimes X(A)^{\otimes m}, X(A)^{\otimes m} \otimes X(A)) \cong \U(X(A)^{\otimes m+1}) \cong \prod_{i,j=1}^{|V|}\U(\mathbb{C}^{A_{i,j}^{m+1}}).
$$
Hence, $\U(X(A)\otimes X(A)^{\otimes m}, X(A)^{\otimes m} \otimes X(A))$ is path connected. Thus, let $U : [0,1] \rightarrow \U(X(A)\otimes X(A)^{\otimes m}, X(A)^{\otimes m} \otimes X(A))$ be a path with $U(0) = 1_{X(A)^{\otimes m+1}}$ and $U(1) = \Phi$. We define a homotopy of $1$-arrows $(H_{U},\Phi^{U}, h_0^{U}, h_1^{U})$ as follows. As a $c(V) - C(I,c(V))$ correspondence we take $H_{U} = C(I, X(A)^{\otimes m})$, and $h_0^{U}, h_1^{U} : X(A)^{\otimes m} \rightarrow X(A)^{\otimes m}$ as merely the identity maps on $X(A)^{\otimes m}$. Finally, we define $\Phi^{U} : X(A) \otimes_{c(V)} C(I,X(A)^{\otimes m}) \rightarrow C(I,X(A)^{\otimes m}) \otimes_{C(I,c(V))} C(I,X(A))$ by the formula $\Phi^{U}_t[f] = U_t[f(t)]$, where we identify $X(A) \otimes_{c(V)} C(I,X(A)^{\otimes m})$ and $C(I,X(A)^{\otimes m}) \otimes_{C(I,c(V))} C(I,X(A))$ with $C(I,X(A)^{\otimes m+1})$. It is clear that $\Phi^{U}_0 = 1_{X(A)^{\otimes m+1}}$ and that $\Phi^{U}_1 = \Phi$, so that the tuple $(H^{U},\Phi^{U}, h_0^{U},h_1^{U})$ is a homotopy equivalence of $1$-arrows. 
\end{proof}

Our next order of business is to show that homotopy shift equivalence is preserved by applying the homomorphism $\mathcal{O}^0$. More precisely,

\begin{prp} \label{prp:funct-homot}
Let $(\A,X)$ and $(\B,Y)$ be objects in $\mathfrak{C}^{\mathbb{N}}_{\mathrm{pr}}$, and let $[F,\Phi_F]$ and $[G,\Phi_G]$ be $1$-arrows from $(\B,Y)$ to $(\A,X)$. If $[F,\Phi_F] \sim [G,\Phi_G]$, then $[\O_{F,\Phi_F}^0, \widehat{\Phi}_F] \sim [\O_{G,\Phi_G}^0, \widehat{\Phi}_G]$.
\end{prp}

\begin{proof}
Recall that for any object $(\C,Z)$ and $t\in I$ we have the $1$-arrow $\ev_t:= [\C_t,\eta_t] : (C(I,\C),C(I,Z)) \leftarrow (\C,Z)$, where $\eta_t$ is the canonical unitary correspondence isomorphism $C(I,Z) \otimes_{C(I,\C)} \C_t \rightarrow \C_t \otimes_{\C} Z$. 

By definition of homotopy between $[F,\Phi_F]$ and $[G,\Phi_G]$ there is a $1$-arrow $[H,\Phi_H] : (\A,X) \leftarrow (C(I,\B), C(I,Y))$ together with $2$-arrows $h_0 : [H_0, \Phi_{H_0}] \rightarrow [F,\Phi_F]$ and $h_1 : [H_1, \Phi_{H_1}] \rightarrow [G,\Phi_G]$, where by definition $[H_0,\Phi_{H_0}] = [H, \Phi_H] \otimes \ev_0$ and $[H_1,\Phi_{H_1}] = [H, \Phi_H] \otimes \ev_1$. This can be described by the following diagram,
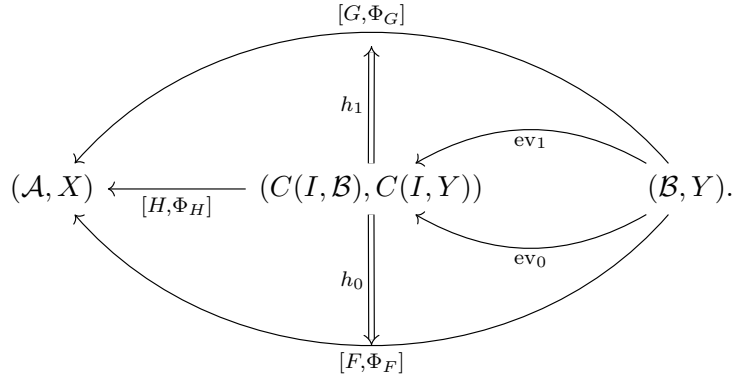
\[
            \begin{tikzcd}
                  {(\A, X)} &  & {(C(I, \B), C(I, Y))} \arrow[ll, "{[H, \Phi_{H}]}"]   &  & {(\B, Y)}. \arrow[ll, "\ev_1", bend right] \arrow[ll, "\ev_0", bend left] \arrow[llll, "{[G,\Phi_G]}"', bend right=49, ""name=UG] \arrow[llll, "{[F, \Phi_F]}", bend left=49, ""name=BF] \\
                  \arrow[from=1-3, to=UG, "h_1", Rightarrow]
                  \arrow[from=1-3, to=BF, "h_0"', Rightarrow]
            \end{tikzcd}
      \]
Thus, the homomorphism $\mathfrak{C}^{\mathbb{N}}_{\mathrm{pr}} \rightarrow \mathfrak{C}^{\mathbb{N}}_{\mathrm{pr}, *}$ from \cite[Corollary 4.7]{MS19} yields the following diagram,
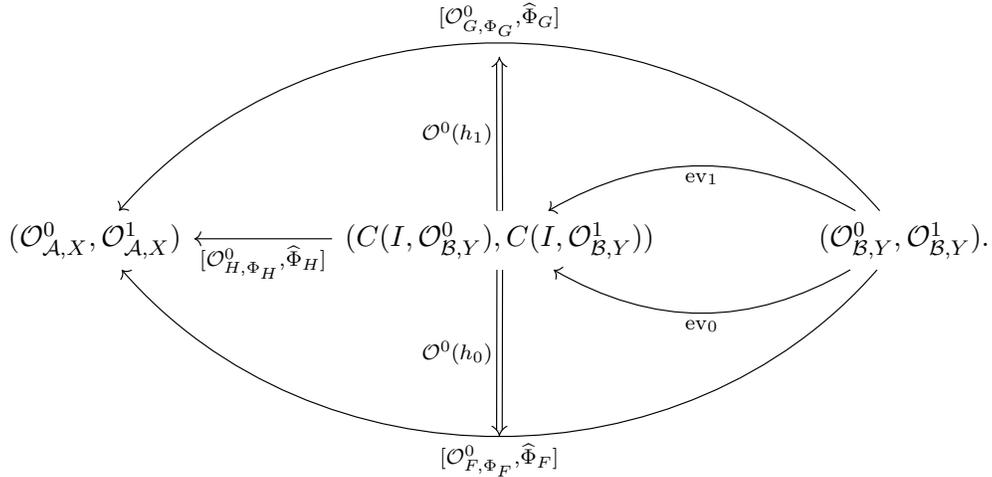
\[
            \begin{tikzcd}
                  {(\O_{\A, X}^0, \O_{\A, X}^1)} &  & {(C(I, \O_{\B, Y}^0), C(I, \O_{\B, Y}^1))} \arrow[ll, "{[\O^0_{H, \Phi_{H}}, \widehat{\Phi}_{H}]}"]   &  & {(\O_{\B, Y}^0, \O_{\B, Y}^1)}. \arrow[ll, "\ev_1", bend right] \arrow[ll, "\ev_0", bend left] \arrow[llll, "{[\O^0_{G,\Phi_G}, \widehat{\Phi}_G]}"', bend right=49, ""name=UG] \arrow[llll, "{[\O^0_{F,\Phi_F}, \widehat{\Phi}_F]}", bend left=49, ""name=BF] \\
                  \arrow[from=1-3, to=UG, "\O^0(h_1)", Rightarrow]
                  \arrow[from=1-3, to=BF, "\O^0(h_0)"', Rightarrow]
            \end{tikzcd}
      \]
Hence, due to the identification $(\O_{C(I, \B), C(I, Y)}^0, \O_{C(I, \B), C(I, Y)}^1) \cong (C(I, \O_{\B, Y}^0), C(I, \O_{\B, Y}^1))$ we see that the above diagram yields a homotopy between $[\O^0_{F,\Phi_F}, \widehat{\Phi}_F]$ and $[\O^0_{G,\Phi_G}, \widehat{\Phi}_G]$
\end{proof}

We now get that if $(\A,X)$ and $(\B,Y)$ are C*-correspondences that are homotopy shift equivalent, then $(\O_X^0,\O_X^1)$ and $(\O_Y^0,\O_Y^1)$ are also homotopy shift equivalent.

\begin{thm} \label{thm:funct-cuntz}
Let $(\A,X)$ and $(\B,Y)$ be objects in $\mathfrak{C}^{\mathbb{N}}_{\mathrm{pr}}$. If $(\A,X)$ and $(\B,Y)$ are aligned / homotopy shift equivalent, then $(\O_X^0,\O_X^1)$ and $(\O_Y^0,\O_Y^1)$ are aligned / homotopy shift equivalent respectively.
\end{thm}

\begin{proof}
That aligned shift equivalence of $(\A,X)$ and $(\B,Y)$ implies that of $(\O_X^0,\O_X^1)$ and $(\O_Y^0,\O_Y^1)$ follows from the discussion preceding Lemma \ref{lem:ind-iso-mor}, equation \eqref{eq:canonical-iso-bicat-imp} and Lemma \ref{lem:ident-obj-cuntz}. That is, if $[M,\Phi_M]$ and $[N,\Phi_N]$ define a concrete aligned shift of lag $m$, then the $1$-arrows $[\O^0_{M,\Phi_M},\widehat{\Phi_M}]$ and $[\O^0_{N,\Phi_N},\widehat{\Phi_N}]$ define a concrete aligned shift of lag $m$ between $(\O_X^0,\O_X^1)$ and $(\O_Y^0,\O_Y^1)$.

As for homotopy shift equivalence, suppose $[M,\Phi_M]$ and $[N,\Phi_N]$ define a homotopy shift of lag $m$ between $(\A,X)$ and $(\B,Y)$. Hence, we have the homotopies $[M,\Phi_M]\otimes [N,\Phi_N] \sim [X^{\otimes m}, 1_{X^{\otimes m+1}}]$ and $[N,\Phi_N]\otimes [M,\Phi_M] \sim [Y^{\otimes m}, 1_{Y^{\otimes m+1}}]$. By equation \eqref{eq:canonical-iso-bicat-imp}, Proposition \ref{prp:funct-homot} and Lemma \ref{lem:ident-obj-cuntz} we get that 
$$
[\O^0_{M,\Phi_M},\widehat{\Phi}_M] \otimes [\O^0_{N,\Phi_N}, \widehat{\Phi}_N] \sim [\O^0_{X^{\otimes m}, 1_{X^{\otimes m+1}}}, \widehat{1}_{X^{\otimes m+1}}] \cong [\O_X^m, 1_{\O_X^{m+1}}]
$$ 
and 
$$
[\O^0_{N,\Phi_N},\widehat{\Phi}_N] \otimes [\O^0_{M,\Phi_M}, \widehat{\Phi}_M] \sim [\O^0_{Y^{\otimes m}, 1_{Y^{\otimes m+1}}}, \widehat{1}_{Y^{\otimes m+1}}] \cong [\O_Y^m, 1_{\O_Y^{m+1}}].
$$ 
Since $(\O_X^1)^{\otimes m} \cong \O_X^m$ and $(\O_Y^1)^{\otimes m} \cong \O_Y^m$, we obtain a homotopy shift of lag $m$ between $(\O_X^0,\O_X^1)$ and $(\O_Y^0,\O_Y^1)$ as required.
\end{proof}

\begin{rmk}
Theorem \ref{thm:funct-cuntz} can be regarded as a corrected version of \cite[Theorem 5.8]{KK-JFA2014} (see their erratum). The mistake in the proof of \cite[Theorem 5.8]{KK-JFA2014} was explained in great detail in \cite{CDE24}, and now Theorem \ref{thm:funct-cuntz} provides a way of fixing it up to homotopy equivalence. In the proof of \cite[Theorem 5.8]{KK-JFA2014} the $1$-arrows $[M,\Phi_M]$ and $[N,\Phi_N]$ were merely assumed to satisfy $M\otimes_{\B}N \cong X^{\otimes m}$ and $N \otimes_{\A} M \cong Y^{\otimes m}$ without reference to $\Phi_M$ and $\Phi_N$. Although there are some maps $\Phi_X$ and $\Phi_Y$ satisfying $[M,\Phi_M] \otimes [N,\Phi_N] \cong [X^{\otimes m}, \Phi_X]$ and $[N,\Phi_N] \otimes [M,\Phi_M] \cong [Y^{\otimes m}, \Phi_Y]$, in general these maps $\Phi_X$ and $\Phi_Y$ cannot both be chosen to be $1_{X^{\otimes m+1}}$ and $1_{Y^{\otimes m+1}}$ up to $2$-arrows. Hence, $\mathcal{O}^0_{X^{\otimes m}, \Phi_X}$ or $\mathcal{O}^0_{Y^{\otimes m}, \Phi_Y}$ may fail to be unitarily isomorphic to $\mathcal{O}_X^m$ or $\mathcal{O}_Y^m$ respectively, and the C*-correspondences $\mathcal{O}^0_{M,\Phi_M}$ and $\mathcal{O}^0_{N,\Phi_N}$ will then fail to yield a shift equivalence between $(\O_X^0,\O_X^1)$ and $(\O_Y^0,\O_Y^1)$.
\end{rmk}

\section{Correspondences arising from $*$-homomorphisms} \label{s:cor-hom}

In this section we describe certain correspondences over stabilized C*-algebras as homomorphism correspondences. This viewpoint allows us to interpret tensor product of correspondences as mere composition of $*$-homomorphisms, and construct $*$-homomorphisms between crossed products by $\mathbb{Z}$. This will be important later on, because up to equivariant $*$-isomorphism Cuntz-Pimsner C*-algebras of imprimitivity bimodules are crossed products.

Let $\A$ and $\B$ be C*-algebras. By \cite[Theorem~2.2]{Lan95}, there exists an isomorphism $\lambda_\B$ from $\M(\B)$ to $\L(\B)$ defined by $\lambda_\B(b)= \mu_b$ with $\mu_b \in \L(\B)$ the left action by the multiplier $b\in \M(\B)$. Furthermore, we have $\lambda_\B(\B) = \K(\B)$ and thus $\lambda_\B|_{\B} \colon B \to \K(\B)$ is a $*$-isomorphism. Now, let $f : \A \rightarrow \B$ be a \emph{non-degenerate} $*$-homomorphism in the sense that $f(\A)\B$ is dense in $\B$. Then $f$ induces a correspondence structure on $\B$ as a right Hilbert C*-module over itself by defining a left action through $f$ via $a\cdot b := f(a) b$. We will denote the resulting correspondence by $_f\B$. We now argue that every proper $\A-\B$-correspondence that is unitarily isomorphic to $\B$ as a right Hilbert C*-module arises in this way.

Recall that in this paper we consider non-degenerate C*-correspondences. Thus, suppose $X$ is a (non-degenerate) proper $\A-\B$ correspondence such that $X$ is unitarily isomorphic to $\B$ as a right Hilbert C*-module. Since $X$ is a non-degenerate proper $\A-\B$ correspondence, its left action induces a non-degenerate $*$-homomorphism from $\A$ to $\K(X)$. Let $\psi  \colon X \to \B$ be a unitary right Hilbert C*-module isomorphism.  Then $\mathrm{Ad}_{\psi}$ is a $*$-isomorphism from $\K(X)$ to $\K(\B)$.  Define $f \colon \A \rightarrow \B$ as the composition 
$$
\A \longrightarrow \K(X) \overset{\mathrm{Ad}_{\psi} }{\longrightarrow} \K(\B) \overset{(\lambda_\B|_{\B})^{-1}}{\longrightarrow} \B,
$$
which is clearly non-degenerate. A straightforward computation shows that the unitary right Hilbert C*-module isomorphism $\psi \colon X \to {}_f\B$ becomes a unitary correspondence isomorphism.

\begin{lem}\label{l:Psi-fg}
  Let $f\colon \A \to \B$ and $g\colon \B \to \C$ be two non-degenerate $*$-homomorphisms. 
  Then, there is a unitary correspondence isomorphism $\Psi = \Psi_{f,g}: {}_f\B \otimes_{\B} {}_g\C \cong {}_{g\circ f}\C$ given by $\Psi(b\otimes c) = g(b)c \in {}_{g\circ f}\C$ for all $b\in \B$ and $c\in \C$.
  The inverse of $\Psi$ is given for all $c\in {}_{g\circ f} \C$ by $\Psi^{-1}(c) = \lim_{\lambda\in \Lambda} e_{\lambda} \otimes c$ for any approximate identity $\{e_{\lambda}\}_{\lambda\in \Lambda}$ in $\B$.
\end{lem}
\begin{proof}
  The map $\Psi$ is correctly defined by functoriality of the internal tensor product. We also have 
\[
  \langle \Psi(b\otimes c), \Psi(b'\otimes c') \rangle = \langle g(b)c, g(b')c' \rangle = c^* g(b^* b) c' = \langle c, \langle b, b \rangle c'\rangle = \langle b\otimes c, b'\otimes c' \rangle
\]
for all $b, b' \in {}_f\B$ and $c, c' \in {}_g\C$.  
Hence, $\Psi$ is an isometric Hilbert $\C$-module map.
Moreover, we have $a\cdot \Psi(b\otimes c) = g(f(a))g(b)c = g(f(a)b)c = \Psi(a\cdot b \otimes c)$ for all $a\in \A$, $b\in \B$ and $c\in \C$. Therefore, $\Psi$ is also a left module map.

Thus, to show that $\Psi$ is a unitary correspondence isomorphism, we need only show that $\Psi$ is surjective. Let $c\in {}_{g\circ f}\C$ be an arbitrary element and fix some approximate identity $\{e_{\lambda}\}_{\lambda\in \Lambda}$ in $\B$. Since $g$ is non-degenerate, $g(e_{\lambda})$ is an approximate identity in $\C$, so that
\[
  \lim_{\lambda\in \Lambda} \Psi(e_{\lambda}\otimes c) = \lim_{\lambda\in \Lambda} g(e_{\lambda})c = c.
\]  
Thus, the formula for $\Psi^{-1}$ defines the inverse of $\Psi$ and we get that $\Psi$ is a unitary correspondence isomorphism.
\end{proof}

Let $u$ be a unitary in the multiplier algebra of $\B$.  Define $M_u \colon \B \to \B$ by $M_u(b)=ub$. It is clear that $M_u$ is a unitary right module isomorphism.  Now, if $f,g : \A\rightarrow \B$ are two non-degenerate $*$-homomorphisms and $u$ is a unitary in the multiplier algebra of $\B$ such that $f(a) = u^*g(a)u$ for all $a\in \A$, then the unitary right isomorphism $M_u : {}_f\B \rightarrow {}_g\B$ is a unitary correspondence isomorphism. Indeed, since $a\cdot M_u(b) = g(a)ub = uf(a)b = M_u(a\cdot b)$ for all $a\in \A$ and $b\in \B$, we see that $M_u$ is a unitary correspondence isomorphism.  

We next show that $M_u : {}_f\B \rightarrow {}_g\B$ is a unitary correspondence isomorphism exactly when $u$ intertwines $f$ and $g$.

\begin{lem}\label{l:hom-u}
Let $f,g : \A \rightarrow \B$ be two non-degenerate $*$-homomorphisms and let $u$ be a unitary in the multiplier algebra of $\B$.  Then $M_u \colon {}_f \B \to {}_g \B$ is a unitary correspondence isomorphism if and only if $f(a) = u^*g(a)u$ for all $a\in \A$.
\end{lem}

\begin{proof}
We have already seen that if $f(a) = u^*g(a)u$ for all $a\in \A$, then $M_u \colon {}_f \B \to {}_g \B$ is a unitary correspondence isomorphism, so we prove the converse. Assume $M_u \colon {}_f \B \to {}_g \B$ is a unitary correspondence isomorphism.  Then for all $a \in \A$ and for all $b \in \B$,
$$
u f(a) b = M_u ( a \cdot b ) = a \cdot M_u(b) = g(a) ub.
$$
Consequently, $u f(a) = g(a)u$ for all $a \in A$ which implies $u f(a) u^* = g(a)$ for all $a \in \A$.
\end{proof}

Our next step is to show that every unitary correspondence isomorphism from ${}_f \B$ to ${}_g \B$ is equal to $M_u$ for some unitary $u$ in the multiplier algebra of $\B$ satisfying the intertwiner condition.

\begin{lem} \label{l:hom-cor-iso}
Let $f,g : \A \rightarrow \B$ be two non-degenerate $*$-homomorphisms. If $\Phi : {}_f\B \rightarrow {}_g\B$ is a unitary correspondence isomorphism. Then there exists $u$ in the multiplier algebra of $\B$ such that $\Phi = M_u$. Hence, ${}_f\B$ and ${}_g\B$ are unitarily isomorphic if and only if $f$ and $g$ are inner conjugate.
\end{lem}
\begin{proof}
Consider the isomorphism $\lambda_\B \colon \M(\B) \to \L(\B)$.  Since $\Phi$ is also a unitary right $\B$-module isomorphism, $\Phi$ is a unitary in $\L(\B)$.  Then $u = \lambda_\B^{-1}(\Phi)$ is a unitary in $\M(\B)$.  

Let $b \in \B$.  To show that $\Phi(b) = M_u(b)$, we will instead show that $\lambda_\B( \Phi(b) ) = \lambda_\B( M_u(b))$. Consider the following,
$$
\lambda_\B( M_u(b))= \lambda_\B (ub) = \lambda_\B(u) \circ \lambda_\B(b) = \Phi \circ \mu_b = \mu_{ \Phi(b) } = \lambda_\B( \Phi(b)). 
$$
Thus, $\Phi = M_u$. By Lemma~\ref{l:hom-u}, as $\Phi$ is a left $\A$ module map, $u f(a) u^* = g(a)$ for all $a \in \A$. Thus, we see that $_f\B$ and $_g\B$ are unitarily isomorphic as C*-correspondences if and only if $f$ and $g$ are inner conjugate.
\end{proof}

\begin{lem} \label{lem:induce-*hom}
Let $\A$ and $\B$ be \emph{stable} and $\sigma$-unital C*-algebras. Suppose that $F$ is a full and proper $\A-\B$ correspondence.
Then $F$ is isomorphic to $_f\B$ for some non-degenerate $*$-homomorphism $f : \A\rightarrow \B$.
\end{lem}

\begin{proof}
Since $\A$ is $\sigma$-unital and the $*$-homomorphism $\varphi_F\colon \A \to \K(X)$ is non-degenerate, $\K(X)$ is also $\sigma$-unital. By \cite[Proposition 6.7]{Lan95}, $F$ is countably generated.

Since $\A$ is stable, we have $\A \cong \A^{\oplus \infty}$ as right Hilbert $\A$-modules. Therefore, we have a unitary isomorphism of right Hilbert $\B$-modules $F \cong \A\otimes F \cong (\A^{\oplus \infty})\otimes F \cong (\A \otimes F)^{\oplus \infty} \cong F^{\oplus \infty}$. By the Brown--Kasparov Stabilization Theorem (see \cite[II.7.6.12]{Bla06}) we have that $F$ is unitarily isomorphic as a right Hilbert $\B$-module to $\B^{\infty} \cong \B$. Thus, by the discussion preceding Lemma \ref{l:Psi-fg} we get that $F$ is of the form $_f \B$ for some non-degenerate $*$-homomorphism $f: \A \rightarrow \B$.
\end{proof}

Now let $(\A,{}_{\alpha}\A)$ and $(\B,{}_{\beta}\B)$ be two correspondences induced from $*$-automorphisms $\alpha$ and $\beta$ on $\A$ and $\B$ respectively. Let $f :\A \rightarrow \B$ be a non-degenerate $*$-homomorphism, which provides us with an $\A-\B$ correspondence $_f \B$. 
Let $u$ be a unitary in the multiplier algebra of $\B$ satisfying $\beta(f(a)) = uf(\alpha(a))u^*$ for all $a\in \A$, so that it defines a unitary correspondence isomorphism $M_{u}\colon {}_{f\circ \alpha}\B\to {}_{\beta\circ f}\B$  by Lemma \ref{l:hom-u}.
Define a map $\Phi_u = \Psi^{-1} \circ M_u \circ \Psi\colon {}_\alpha \A \otimes {}_f \B \to {}_f \B \otimes {}_\beta \B$, where $\Phi$ denotes the isomorphism from Lemma~\ref{l:Psi-fg}. Then $\Phi_u$ is a unitary correspondence isomorphism which defines the $1$-arrow $[{}_f\B,\Phi_u]\colon (\A, {}_{\alpha}\A) \leftarrow (\B,{}_{\beta}\B)$.

On the other hand, if $[{}_f\B, \Phi] \colon (\A, {}_{\alpha}\A) \leftarrow (\B,{}_{\beta}\B)$ is a $1$-arrow, then $\Psi \circ \Phi \circ \Psi^{-1} = M_u$ for some unitary $u$ in the multiplier algebra of $\B$ by Lemma~\ref{l:hom-cor-iso}.  Thus, every 1-arrow is of the form $[{}_f\B, \Phi_u]$ for some $*$-homomorphism $f$ and unitary $u$ in the multiplier algebra of $\B$. 

\begin{lem} \label{lem:2-arr-calc}
Let $[{}_f\B, \Phi_u]: (\A, {}_{\alpha}\A) \leftarrow (\B, {}_{\beta}\B)$ and $[{}_g\C,\Phi_v] : (\B, {}_{\beta}\B) \leftarrow (\C,{}_{\gamma}\C)$ be composable $1$-arrows. Then the canonical unitary correspondence isomorphism $\Psi: {}_f\B \otimes_{\B} {}_g\C \rightarrow {}_{g\circ f}\C$ is a $2$-arrow between the $1$-arrows $ [{}_f\B, \Phi_u] \circ [{}_g\C,\Phi_v]$ and $[{}_{g\circ f}\C, \Phi_{vg(u)}]$.
\end{lem}

\begin{proof}
Fix an approximate identity $\{e_{\lambda}\}_{\lambda \in \Lambda}$ in $\A$. Since $f$ and $g$ are non-degenerate, we get that $e_{\lambda}' = f(e_{\lambda})$ is an approximate identity in $\B$ and $e''_{\lambda} = g(f(e_{\lambda}))$ is an approximate identity in $\C$. We need to prove that the following diagram commutes.
\[
\begin{tikzcd}
{}_\alpha \A \otimes  {}_f \B \otimes  {}_g \C \arrow[r, "\Phi_u\otimes 1_{\C}"] \arrow[d, "1_\A \otimes \Psi"] & {}_f \B \otimes   {}_\beta \B \otimes {}_g \C \arrow[r, "1_\B \otimes \Phi_v"] & {}_f \B \otimes  {}_g \C \otimes {}_\gamma \C \arrow[d, "\Psi\otimes 1_\C"] \\
{}_\alpha \A \otimes {}_{g\circ f} \C \arrow[rr, "\Phi_{vg(u)}"]                                                &                                                                                & {}_{g\circ f} \C \otimes {}_\gamma \C                                      
\end{tikzcd}.
\]

Consider an arbitrary element $a\otimes b\otimes c$ in ${}_\alpha \A \otimes  {}_f \B \otimes  {}_g \C$.
We will chase the element through the diagram by expanding $\Phi_u$ and $\Phi_v$ by definition and by using the inverse formula for $\Psi$ from Lemma~\ref{l:Psi-fg}. 
We have 
\[
  (\Phi_u\otimes 1_{\C})(a\otimes b\otimes c) = \Psi^{-1}(u f(a)b) \otimes c = \lim_{\lambda\in \Lambda} e'_{\lambda} \otimes u f(a)b \otimes c.
\]
Analogously, we have
\[
  (1_{\B} \otimes \Phi_v)(e'_\lambda \otimes u f(a) b \otimes c) = \lim_{\mu\in \Lambda} e'_\lambda \otimes e''_\mu \otimes v g(u) g(f(a)b) c.
\]
Finally, we compute 
\[
  (\Psi\otimes 1_{\C})(e'_{\lambda} \otimes e''_{\mu} \otimes v g(u) g(f(a)b) c) = e''_{\lambda}e''_{\mu} \otimes vg(u) g(f(a)b) c
\]  
and, consequently, the image of $a\otimes b\otimes c$ under the upper path of the diagram is
\[
  \lim_{\mu}\lim_{\lambda} e''_{\lambda}e''_{\mu} \otimes vg(u) g(f(a)b) c = \lim_{\lambda} e''_{\lambda} \otimes vg(u) g(f(a)b) c.
\]
On the other hand, we have
\[
  \Phi_{vg(u)} \circ (1_{\A} \otimes \Psi)(a\otimes b\otimes c) = \Phi_{vg(u)}(a\otimes  g(b)c) =\lim_{\lambda} e''_{\lambda}\otimes vg(u) g(f(a)b) c,
\]  
so the diagram commutes.
\end{proof}

For what follows, we assume some familiarity with the theory of crossed products, as is fleshed out in \cite{RW-Morita-Eq}. Given a $1$-arrow $[{}_f\B, \Phi_u] : (\A,{}_{\alpha}\A) \leftarrow (\B,{}_{\beta}\B)$, we may define a $*$-homomorphism $f_{\rtimes u} : \A \rtimes_{\alpha} \mathbb{Z} \rightarrow \B \rtimes_{\beta} \mathbb{Z}$ by the rules $f_{\rtimes u}|_{\A} = f$ and $f_{\rtimes u}(S_{\alpha}a) = S_{\beta}uf(a)$, where $S_{\alpha}$ and $S_{\beta}$ are the unitary generators of the respective crossed products. It is straightforward to see that this defines a $*$-homomorphism which is equivariant with respect to the dual $\mathbb{T}$ action, that $f_{\rtimes u}|_{\A} = f$, and that $f_{\rtimes u}$ is an isomorphism if and only if $f$ is.

\begin{prp} \label{p:func-cros-prod}
The assignments $(\A,{}_{\alpha}\A) \mapsto \A\rtimes_{\alpha} \mathbb{Z}$ and $[{}_f\B, \Phi_u] \mapsto f_{\rtimes u}$ satisfies the following composition rule: if $[{}_f\B, \Phi_u] : (\A,{}_{\alpha}\A) \leftarrow (\B,{}_{\beta}\B)$ and $[{}_g\C, \Phi_v] : (\B,{}_{\beta} \B) \leftarrow (\C,{}_{\gamma}\C)$ are composable $1$-arrows, then $g_{\rtimes v} \circ f_{\rtimes u} = (g\circ f)_{\rtimes vg(u)}$.
\end{prp}
\begin{proof}
Clearly we have $(g_{\rtimes v} \circ f_{\rtimes u})|_{\A} = g\circ f = (g\circ f)_{\rtimes vg(u)} |_{\A}$. Now, since
$$
g_{\rtimes v}(f_{\rtimes u}(S_{\alpha})a) = g_{\rtimes v} (S_{\beta}uf(a)) = S_{\gamma}vg(u)g(f(u)) = (g\circ f)_{\rtimes vg(u)}(S_{\alpha} a)
$$
for all $a\in \A$, and since $S_{\alpha}$ generates the crossed product $\A \rtimes_{\alpha}\mathbb{Z}$ together with $\A$, we are done.
\end{proof}

\begin{lem}\label{l:2-arr-inner}
  Suppose that $[{}_f \B, \Phi_u], [{}_g \B, \Phi_v] \colon (\A, {}_{\alpha}\A) \leftarrow (\B, {}_{\beta}\B)$ are two $1$-arrows and let $M_w\colon [{}_f \B, \Phi_u] \rightarrow [{}_g \B, \Phi_v]$ be a $2$-arrow coming from a unitary multiplier $w$ of $\B$.
  Then, $w f_{\rtimes u}(a) w^*  = g_{\rtimes v}(a)$ for all $a\in \A \rtimes_\alpha \bbZ$.
\end{lem}
\begin{proof}
  The $2$-arrow $M_w$ is a unitary correspondence isomorphism between ${}_f \B$ and ${}_g \B$.
  Therefore, by Lemma~\ref{l:hom-u} we have $w f(a) w^* = g(a)$ for all $a\in \A$.
  Thus, we need only verify that $w f_{\rtimes u}(S_\alpha a) w^* = g_{\rtimes v}(S_\alpha a)$ for all $a\in \A$.

  By definition of 2-arrow we get the following commuting diagram
  \[
    \begin{tikzcd}
      {}_\alpha \A \otimes {}_f \B \arrow[r, "\Phi_u"] \arrow[d, "1_\A\otimes M_w"] & {}_f \B \otimes {}_\beta \B \arrow[d, "M_w \otimes 1_\B"] \\
      {}_\alpha \A \otimes {}_g \B \arrow[r, "\Phi_v"]                              & {}_g \B \otimes {}_\beta \B                              
    \end{tikzcd}
  \]

  Let $a\otimes b \in {}_\alpha \A \otimes {}_f \B$ be an arbitrary element.
  Then, analogously to the proof of Lemma~\ref{lem:2-arr-calc}, we have 
  \[
    (M_w \otimes 1_\B) (\Phi_u(a\otimes b)) = \lim_{\lambda} e_\lambda \otimes \beta(w)uf(a)b
  \]
  and
  \[
    \Phi_v((1_\A \otimes M_w)(a\otimes b)) = \lim_{\lambda} e_\lambda \otimes v g(a) w b
  \]
  where $\{e_\lambda\}_{\lambda}$ is an approximate identity in $\B$. From the commutativity of the diagram, we get $\beta(w)u f(a) = v g(a)w$ for all $a\in \A$. Thus, we have $wf_{\rtimes u}(S_\alpha a)w^* = w S_\beta u f(a) w^* = S_\beta \beta(w) u f (a) w^* = S_\beta v g(a) = g_{\rtimes v}(S_\alpha a)$.
\end{proof}

The following result shows that homotopy is preserved under the crossed product operation when the associated C*-algebras are stable and $\sigma$-unital. Recall that a $1$-arrow $[X,\Phi]$ is called an equivalence arrow if $X$ is an imprimitivity bimodule.

\begin{prp} \label{p:homotopic-crossed-prod}
Let $(\A,{}_{\alpha}\A)$ and $(\B,{}_{\beta}\B)$ be objects such that $\A$ and $\B$ are stable and $\sigma$-unital, whereas $\alpha$ and $\beta$ are $*$-automorphisms on $\A$ and $\B$ respectively. Suppose $[{}_f\B, \Phi_u] : (\A,{}_{\alpha}\A) \leftarrow (\B,{}_{\beta}\B)$ is a $1$-arrow between them. If $[{}_f\B, \Phi_u]$ is homotopic to an equivalence 1-arrow via a homotopy $(H,\Phi_H,h_0,h_1)$, then $f_{\rtimes u}$ is equivariantly homotopic to a $*$-isomorphism.
\end{prp}
\begin{proof}
  Let $[X,\Phi]$ be an equivalence $1$-arrow. Since $\A \cong \K(X)$ is $\sigma$-unital, we get that $X$ is countably generated by \cite[Proposition 6.7]{Lan95}. Since $\B$ is $\sigma$-unital and $X$ is full, by Lemma \ref{lem:induce-*hom} and Lemma \ref{l:hom-cor-iso} there is a $*$-isomorphism $g: \A \rightarrow \B$ and a unitary multiplier $v$ of $\B$ such that $[X, \Phi] \cong [{}_g\B, \Phi_v]$ as 1-arrows. 
Moreover, $g_{\rtimes v}$ is an equivariant $*$-isomorphism since its restriction $g$ onto the fixed point algebra $\A$ is a $*$-isomorphism.

Now let $(H,\Phi_H, h_0,h_1)$ be a homotopy equivalence between $[{}_f\B,\Phi_u]$ and $[X,\Phi] \cong [{}_g\B,\Phi_v]$. By Lemma \ref{lem:induce-*hom} and Lemma \ref{l:hom-cor-iso} we may write it as $H \cong {}_{\tau}C(I,\B)$ for a $*$-homomorphism $\tau :\A \rightarrow C(I,\B)$, and $\Phi_H = \Phi_U$ for some unitary multiplier $U$ of $C(I,\B)$. This means that $h_0 = M_{w_0}$ is a $2$-arrow between $[{}_{\tau_0}\B,\Phi_{U_0}]$ and $[{}_f \B,\Phi_u]$ and that $h_1 = M_{w_1}$ is a $2$-arrow between $[{}_{\tau_1}\B,\Phi_{U_1}]$ and $[{}_g\B,\Phi_v]$.
By Lemma~\ref{l:2-arr-inner}, we have $\mathrm{Ad}_{w_0}((\tau_0)_{\rtimes U_0}(a)) = f_{\rtimes u}(a)$ and $\mathrm{Ad}_{w_1}((\tau_1)_{\rtimes U_1}(a)) = g_{\rtimes v}(a)$ for all $a\in \A \rtimes_{\alpha} \mathbb{Z}$. 
In particular, $(\tau_1)_{\rtimes U_1}$ is an equivariant $*$-isomorphism since $g_{\rtimes v}$ is an equivariant $*$-isomorphism.

The $1$-arrow $[H,\Phi_U]$ induces an equivariant $*$-homomorphism $\tau_{\rtimes U} : \A \rtimes_{\alpha} \mathbb{Z} \rightarrow C(I,\B)\rtimes_{\id_{C(I)}\otimes \beta} \mathbb{Z} \cong C(I, \B\rtimes_{\beta}\mathbb{Z})$ which then gives rise to an equivariant homotopy between $(\tau_0)_{\rtimes U_0}$ and the equivariant $*$-isomorphism $(\tau_1)_{\rtimes U_1}$. Thus, after adjoining the $*$-homomorphism $\tau_{\rtimes U}$ with $w_0$, we obtain an equivariant homotopy of $f_{\rtimes u}$ with a $*$-isomorphism.
\end{proof}

\section{Classification up to equivariant homotopy equivalence} \label{s:main}

In this section we prove that shift equivalence of finite essential matrices $A$ and $B$ coincides with equivariant homotopy equivalence of $C^*(G_A) \otimes \mathbb{K}$ and $C^*(G_B)\otimes \mathbb{K}$, and that aligned shift equivalence of $X(A)$ and $X(B)$ implies that $C^*(G_A) \otimes \mathbb{K}$ and $C^*(G_B)\otimes \mathbb{K}$ are equivariantly isomorphic. We will assume some rudimentary familiarity with K-theory and crossed products for graph C*-algebras, which can be found in \cite{LLR00} and \cite[Chapter 7]{Raeburn-graph}.

Let $A$ be a finite essential matrix with entries in $\mathbb{N}$. We denote by $(D_A,D_A^+)$ the inductive limit of the following system of ordered abelian groups
\[\begin{CD}
(\bbZ^V,\bbZ^V_+) @>A^T>> (\bbZ^V,\bbZ^V_+) @>A^T>> (\bbZ^V,\bbZ^V_+) @>A^T>> \cdots
\end{CD}\]
and let $d_A : D_A \rightarrow D_A$ be the homomorphism induced by
\[\begin{CD}
	\bbZ^V @>A^T>> \bbZ^V @>A^T>> \bbZ^V @>A^T>> \cdots D_A \\
	@VVA^T V @VVA^T V @VVA^T V @VV d_AV  \\
	\bbZ^V @>A^T>> \bbZ^V @>A^T>> \bbZ^V @>A^T>> \cdots D_A.
\end{CD}\]
The triple $(D_A,D_A^+,d_A)$ is called the \emph{dimension triple} of $A$. The importance of this dimension triples comes from a theorem of Krieger \cite[Theorem 6.4]{Eff81} (see also \cite{Kri80}), which states that two essential matrices $A$ and $B$ are shift equivalent if and only if they have isomorphic dimension triples. In fact, this data can be recovered from the underlying graph C*-algebras together with the gauge action, and coincides with the $K$-theory triple of the crossed product by the gauge action $\gamma^A$ on $C^*(G_A)$. 
It follows from \cite[Corollary 7.14]{Raeburn-graph} together with the discussion preceding \cite[Lemma 7.15]{Raeburn-graph} that the triple $(D_A,D_A^+, d_A)$ is isomorphic to the triple, 
$$
(K_0(C^*(G_A)\rtimes_{\gamma^A}\bbT),K_0(C^*(G_A)\rtimes_{\gamma^A}\mathbb{T})^+,K_0(\widehat{\gamma}^A_1)^{-1}),
$$
where $\widehat{\gamma}^A_1$ is the dual $\mathbb{Z}$ action on $C^*(G_A)\rtimes_{\gamma^A}\mathbb{T}$. Through the use of equivariant $K$-theory, this allows us to show that shift equivalence is an invariant of equivariant homotopy equivalence of the associated graph C*-algebras.

\begin{prp} \label{p:eq-hom-eq-se}
Let $A$ and $B$ be finite essential matrices, and suppose that $C^*(G_A)\otimes \mathbb{K}$ and $C^*(G_B)\otimes \mathbb{K}$ are equivariantly homotopy equivalent. Then $A$ and $B$ are shift equivalent.
\end{prp}

\begin{proof}
First, by \cite[Corollary 3.1]{ERS22} (see also \cite[Theorem 2.6.1]{Phinotes}) and \cite[Corollary 3.3]{ERS22} we have coincidence of the triple
$$
(K_0(C^*(G_A)\rtimes_{\gamma^A}\bbT),K_0(C^*(G_A)\rtimes_{\gamma^A}\mathbb{T})^+,K_0(\widehat{\gamma}^A_1)^{-1}),
$$
with the equivariant $K$-theory $\mathbb{Z}[x,x^{-1}]$-module $(K_0^{\mathbb{T}}(C^*(G_A) \otimes \mathbb{K}), K^{\mathbb{T}}_0(C^*(G_A) \otimes \mathbb{K})^+)$. 

Using equivariant homotopy equivalence we apply \cite[Proposition 2.5.10]{Phinotes} and its proof to get that the $\mathbb{Z}[x,x^{-1}]$-modules 
\begin{align*}
(K_0^{\mathbb{T}}(C^*(G_A) \otimes \mathbb{K}), K^{\mathbb{T}}_0(C^*(G_A) \otimes \mathbb{K})^+) \\
(K_0^{\mathbb{T}}(C^*(G_B) \otimes \mathbb{K}), K^{\mathbb{T}}_0(C^*(G_B) \otimes \mathbb{K})^+)
\end{align*}
are isomorphic as $\mathbb{Z}[x,x^{-1}]$-modules. Therefore, the triples 
\begin{align*}
(K_0(C^*(G_A)\rtimes_{\gamma^A}\bbT),K_0(C^*(G_A)\rtimes_{\gamma^A}\mathbb{T})^+,K_0(\widehat{\gamma}^A_1)^{-1}) \\
(K_0(C^*(G_B)\rtimes_{\gamma^B}\bbT),K_0(C^*(G_B)\rtimes_{\gamma^B}\mathbb{T})^+,K_0(\widehat{\gamma}^B_1)^{-1}),
\end{align*}
are also isomorphic. Since these triples coincide with the respective dimension triples for $A$ and $B$ (see \cite[Corollary 7.14]{Raeburn-graph} together with the discussion preceding \cite[Lemma 7.15]{Raeburn-graph}), we see that $A$ and $B$ are shift equivalent by \cite[Theorem 6.4]{Eff81}.
\end{proof}

\begin{thm} \label{t:main}
Let $A$ and $B$ be finite essential square matrices with entries in $\mathbb{N}$, over $V$ and $W$ respectively. The following are equivalent.
\begin{enumerate}
\item $A$ and $B$ are shift equivalent.
\item $(c(V),X(A))$ and $(c(W),X(B))$ are homotopy shift equivalent.
\item $C^*(G_A) \otimes \mathbb{K}$ and $C^*(G_B) \otimes \mathbb{K}$ are equivariantly homotopy equivalent.
\end{enumerate}
\end{thm}

\begin{proof}
The implication $(3)\Rightarrow (1)$ follows directly from Proposition \ref{p:eq-hom-eq-se}, and the implication $(1) \implies (2)$ follows directly from Proposition \ref{p:se=hse}. 

Thus, we are left with showing $(2) \implies (3)$. Let $(M,N,\Phi_M, \Phi_N)$ be a concrete homotopy shift equivalence of lag $m$ between $(c(V),X(A))$ and $(c(W),X(B))$. Apply Theorem \ref{thm:funct-cuntz} to get that $(\O_{X(A)}^0,\O_{X(A)}^1)$ and $(\O_{X(B)}^0,\O_{X(B)}^1)$ are homotopy shift equivalent via $[\O^0_{M,\Phi_M}, \widehat{\Phi}_M]$ and $[\O^0_{N,\Phi_N}, \widehat{\Phi}_N]$,  where all involved C*-correspondences are proper and full.

We stabilize all involved objects, $1$-arrows and $2$-arrows, and this will preserve homotopy shift equivalence, as well as properness and fullness of all involved C*-correspondences. We abuse notation and will not carry over the tensoring by compacts in our notation and arguments. For instance, we write $\O_{X(A)}^1$ for the correspondence $\O_{X(A)}^1 \otimes \mathbb{K}$, and $[\O^0_{M,\Phi_M}, \widehat{\Phi}_M]$ for the $1$-arrow $[\O^0_{M,\Phi_M}\otimes \mathbb{K}, \widehat{\Phi}_M \otimes \id_{\mathbb{K}}]$.

Now, as all C*-algebras involved are stable and $\sigma$-finite, and all correspondences involved are proper and full, by Lemma \ref{lem:induce-*hom} we have that $\O_{X(A)}^m \cong {}_{\alpha^m}\O_{X(A)}^0$, $\O_{X(B)}^m \cong {}_{\beta^m}\O_{X(B)}^0$ for some $*$-automorphisms $\alpha$ and $\beta$, and $\O^0_{M,\Phi_M} \cong {}_{f}\O_{X(B)}^0$, $\O^0_{N,\Phi_N} \cong {}_{g}\O_{X(A)}^0$ for some $*$-homomorphisms $f : \O^0_{X(A)} \rightarrow \O^0_{X(B)}$ and $g: \O^0_{X(B)} \rightarrow \O^0_{X(A)}$. Moreover, by Lemma \ref{l:hom-cor-iso} we also have that $\widehat{\Phi}_M = \Phi_u$ and $\widehat{\Phi}_N = \Phi_v$ for unitary multipliers $u$ of $\O^0_{X(B)}$ and $v$ of $\O^0_{X(A)}$. Hence, by Lemma \ref{lem:2-arr-calc} we have 
$$
[\O^0_{M,\Phi_M}\otimes \O^0_{N,\Phi_N}, \widehat{\Phi}_M \odot \widehat{\Phi}_N ] \cong [{}_{f}\O_{X(B)}^0 \otimes {}_{g}\O_{X(A)}^0, \Phi_u \odot \Phi_v] \cong [{}_{g \circ f}\O_{X(A)}^0, \Phi_{vg(u)}]
$$
$$[\O^0_{N,\Phi_N}\otimes \O^0_{M,\Phi_M} , \widehat{\Phi}_N \odot \widehat{\Phi}_M] \cong [{}_{g}\O_{X(A)}^0 \otimes {}_{f}\O_{X(B)}^0, \Phi_v \odot \Phi_u] \cong [{}_{f \circ g}\O_{X(B)}^0, \Phi_{f(v)u}].
$$ 
Due to homotopy shift equivalence via $[\O^0_{M,\Phi_M}, \widehat{\Phi}_M]$ and $[\O^0_{N,\Phi_N}, \widehat{\Phi}_N]$, by Propositions \ref{p:func-cros-prod} and \ref{p:homotopic-crossed-prod} we get that $(g\circ f)_{\rtimes vg(u)} = g_{\rtimes v}\circ f_{\rtimes u}$ and $(f\circ g)_{\rtimes f(v)u} = f_{\rtimes u}\circ g_{\rtimes v}$, as $*$-endomorphisms on $\O_{X(A)}^0 \rtimes_{\alpha} \mathbb{Z} \cong \O_{X(A)}$ and $\O_{X(B)}^0 \rtimes_{\beta} \mathbb{Z} \cong \O_{X(B)}$ respectively, are each equivariantly homotopic to $*$-automorphisms $\tau$ on $\O_{X(A)} \cong C^*(G_A)$ and $\rho$ on $\O_{X(B)} \cong C^*(G_B)$ respectively. Hence, $\tau^{-1} \circ g_{\rtimes v}$ is an equivariant homotopy left inverse of $f_{\rtimes u}$, and $g_{\rtimes v} \circ \rho^{-1}$ is an equivariant homotopy right inverse of $f_{\rtimes u}$ and they must therefore be homotopy equivalent by coincidence of left and right inverses in a monoid. Thus, we get that (the stabilizations of) $C^*(G_A)$ and $C^*(G_B)$ are equivariantly homotopy equivalent.
\end{proof}

Finally, we provide a sufficient condition for the existence of an equivariant $*$-isomorphism between $C^*(G_A) \otimes \mathbb{K}$ and $C^*(G_B) \otimes \mathbb{K}$ which is the C*-algebraic analogue of a recent result in Leavitt path algebras \cite[Theorem 6.7]{ART23}. The proof goes along the same lines as that of Theorem \ref{t:main}, but is somewhat simpler.

\begin{thm}
Let $A$ and $B$ be finite essential square matrices with entries in $\mathbb{N}$, over $V$ and $W$ respectively. Suppose that $X(A)$ and $X(B)$ are aligned shift equivalent. Then $C^*(G_A) \otimes \mathbb{K}$ and $C^*(G_B) \otimes \mathbb{K}$ are equivariantly $*$-isomorphic.
\end{thm}

\begin{proof}
Let $(M,N,\Phi_M,\Phi_N,\Psi_{X(A)},\Psi_{X(B)})$ be a concrete aligned shift between $X(A)$ and $X(B)$ of lag $m$. By Theorem \ref{thm:funct-cuntz} we get that $(\mathcal{O}^0_{X(A)}, \mathcal{O}^1_{X(A)})$ and $(\mathcal{O}^0_{X(B)}, \mathcal{O}^1_{X(B)})$ are aligned shift equivalent with lag $m$ via the $1$-arrows $[\O^0_{M,\Phi_M},\widehat{\Phi_M}]$ and $[\O^0_{N,\Phi_N},\widehat{\Phi_N}]$ where $\O^0_{M,\Phi_M}$ and $\O^0_{N,\Phi_N}$ are still proper and full.

We stabilize all involved objects, $1$-arrows and $2$-arrows, and this will preserve aligned shift equivalence, as well as properness and fullness of all involved C*-correspondences. We abuse notation and will not carry over the tensoring by compacts in our notation and arguments. For instance, we write $\O_{X(A)}^1$ for the correspondence $\O_{X(A)}^1 \otimes \mathbb{K}$, and $[\O^0_{M,\Phi_M}, \widehat{\Phi}_M]$ for the $1$-arrow $[\O^0_{M,\Phi_M}\otimes \mathbb{K}, \widehat{\Phi}_M \otimes \id_{\mathbb{K}}]$.

Now, since all C*-algebras involved are stable and $\sigma$-finite, and all correspondences involved are proper and full, by Lemma \ref{lem:induce-*hom} we have that $\O_{X(A)}^m \cong {}_{\alpha^m}\O_{X(A)}^0$, $\O_{X(B)}^m \cong {}_{\beta^m}\O_{X(B)}^0$ for some $*$-automorphisms $\alpha$ and $\beta$, and $\O^0_{M,\Phi_M} \cong {}_{f}\O_{X(B)}^0$, $\O^0_{N,\Phi_N} \cong {}_{g}\O_{X(A)}^0$ for some $*$-homomorphisms $f : \O^0_{X(A)} \rightarrow \O^0_{X(B)}$ and $g: \O^0_{X(B)} \rightarrow \O^0_{X(A)}$. Moreover, by Lemma \ref{l:hom-cor-iso} we also have that $\widehat{\Phi}_M = \Phi_u$ and $\widehat{\Phi}_N = \Phi_v$ for unitary multipliers $u$ of $\O^0_{X(B)}$ and $v$ of $\O^0_{X(A)}$.

By our assumptions together with Lemma \ref{lem:2-arr-calc} we have 
$$[\O_{X(A)}^m,1_{\O_{X(A)}^{m+1}}]\cong [(\O_{X(A)}^1)^{\otimes m},1_{(\O_{X(A)}^1)^{\otimes m+1}}] \cong [\O^0_{M,\Phi_M}\otimes \O^0_{N,\Phi_N}, \widehat{\Phi}_M \odot \widehat{\Phi}_N ] \cong$$
$$[{}_{f}\O_{X(B)}^0 \otimes {}_{g}\O_{X(A)}^0, \Phi_u \odot \Phi_v] \cong [{}_{g \circ f}\O_{X(A)}^0, \Phi_{vg(u)}],
$$
$$[\O_{X(B)}^m,1_{\O_{X(B)}^{m+1}}] \cong  [(\O_{X(B)}^1)^{\otimes m},1_{(\O_{X(B)}^1)^{\otimes m+1}}] \cong [\O^0_{N,\Phi_N}\otimes \O^0_{M,\Phi_M} , \widehat{\Phi}_N \odot \widehat{\Phi}_M] \cong $$
$$[{}_{g}\O_{X(A)}^0 \otimes {}_{f}\O_{X(B)}^0, \Phi_v \odot \Phi_u] \cong [{}_{f \circ g}\O_{X(B)}^0, \Phi_{f(v)u}].
$$ 
In particular, by Lemma \ref{l:hom-cor-iso} we get that $g\circ f$ and $f\circ g$ are inner conjugate to $\alpha^m$ and $\beta^m$ respectively. By Lemma \ref{l:2-arr-inner} and Propositions \ref{p:func-cros-prod} we get that $(\alpha^m)_{\rtimes vg(u)}$ is unitarily equivalent to $(g\circ f)_{\rtimes vg(u)} = g_{\rtimes v}\circ f_{\rtimes u}$ and that $(\beta^m)_{\rtimes f(v)u}$ is unitarily equivalent to $(f\circ g)_{\rtimes f(v)u} = f_{\rtimes u}\circ g_{\rtimes v}$, as equivariant $*$-endomorphisms on $\O_{X(A)}^0 \rtimes_{\alpha} \mathbb{Z} \cong \O_{X(A)}$ and $\O_{X(B)}^0 \rtimes_{\beta} \mathbb{Z} \cong \O_{X(B)}$ respectively. Hence, since $(\alpha^m)_{\rtimes vg(u)}$ and $(\beta^m)_{\rtimes f(v)u}$ are $*$-automorphisms, we readily see that $f_{\rtimes u}$ is an equivariant $*$-isomorphism between (the stabilizations of) $C^*(G_A)$ and $C^*(G_B)$.
\end{proof}

Aligned shift equivalence of the C*-correspondences $X(A)$ and $X(B)$ is still a rather mysterious equivalence relation, and although it is implied by SSE and implies SE, we do not yet know how it is situated between them.

\vspace{6pt}

\textbf{Acknowledgments.}
The authors are grateful to Ralf Meyer for valuable conversations on bicategories conducted at University of G\"ottingen and at Mathematisches Forsch\-ungsinstitut Oberwolfach.

\end{document}